\documentclass[a4paper, 11pt, reqno]{amsart}
\usepackage{mathrsfs}
\usepackage{amsfonts}
\usepackage[centertags]{amsmath}
\usepackage{amssymb}
\usepackage{amsthm}
\usepackage{graphicx}
\usepackage{caption}
\usepackage{hyperref}
\usepackage{enumerate}
\usepackage[textwidth=16cm, hmarginratio=1:1]{geometry}
\usepackage{appendix}
\usepackage[all]{xy}
\usepackage{float}
\usepackage{amsmath,amsfonts,amssymb,graphicx}
\usepackage{longtable}
\usepackage{resizegather}
\usepackage{framed}
\usepackage{adjustbox}
\usepackage{tikz}

\newtheorem{theorem}{Theorem}[section]
\newtheorem{corollary}[theorem]{Corollary}
\newtheorem{lemma}[theorem]{Lemma}
\newtheorem{proposition}[theorem]{Proposition}
\newtheorem{definition}[theorem]{Definition}
\newtheorem{remark}[theorem]{Remark}
\newtheorem{example}[theorem]{Example}

\newtheorem{conjecture}[theorem]{Conjecture}

\numberwithin{equation}{section}

\DeclareMathOperator*{\diag}{diag}

\DeclareMathOperator*{\seq}{Seq}
\DeclareMathOperator*{\res}{Res}

\begin{document}

\title{M-systems and Cluster algebras}
\author{Qian-Qian Zhang, Bing Duan, Jian-Rong Li, Yan-Feng Luo}
\address{Qian-Qian Zhang: School of Mathematics and Statistics, Lanzhou University, Lanzhou 730000, P. R. China.}
\email{Zhangqq2014@lzu.edu.cn}

\address{Bing Duan: School of Mathematics and Statistics, Lanzhou University, Lanzhou 730000, P. R. China.}
\email{duan890818@163.com}

\address{Jian-Rong Li, Einstein Institute of Mathematics, The Hebrew University of Jerusalem, Jerusalem 9190401, Israel, and School of Mathematics and Statistics, Lanzhou University, Lanzhou 730000, P. R. China}
\email{lijr07@gmail.com, lijr@lzu.edu.cn}

\address{Yan-Feng Luo: School of Mathematics and Statistics, Lanzhou University, Lanzhou 730000, P. R. China.}
\email{luoyf@lzu.edu.cn}
\date{}

\maketitle

\begin{abstract}
The aim of this paper is two-fold: (1) introduce four systems of equations called M-systems and dual M-systems of types $A_{n}$ and $B_{n}$ respectively; (2) make a connection between M-systems (dual M-systems) and cluster algebras and prove that the Hernandez-Leclerc conjecture is true for minimal affinizations of types $A_n$ and $B_n$.

\hspace{0.15cm}


\hspace{0.15cm}

\end{abstract}

\section{Introduction}
In the paper \cite{FZ02}, Fomin and Zelevinsky introduced the theory of cluster algebras to study canonical bases of quantum groups introduced by Lusztig \cite{L90} and Kashiwara \cite{K91} and total positivity for semisimple algebraic groups developed by Lusztig \cite{L94}. It has exciting connections and applications to many areas of mathematics and physics including integrable systems, Poisson geometry, quiver representations, Teichm\"{u}ller theory, and tropical geometry.

The aim of this paper is two-fold:
\begin{enumerate}[(1)]
\item introduce four systems of equations called M-systems and dual M-systems of types $A_{n}$ and $B_{n}$ respectively;

\item make a connection between M-systems (dual M-systems) and cluster algebras and prove that the Hernandez-Leclerc conjecture (Conjecture 13.2 in \cite{HL10} and Conjecture 9.1 in \cite{Le10}, Conjecture 5.2 in \cite{HL13}) is true for minimal affinizations of types $A_n$ and $B_n$.
\end{enumerate}

Let $\mathfrak{g}$ be a simple Lie algebra and $U_q \widehat{\mathfrak{g}}$ the corresponding quantum affine algebra. The T-systems are functional relations which were defined in \cite{KNS94}.   Nakajima and Hernandez proved that the $q$-characters of Kirillov-Reshetikhin modules satisfy the T-systems, see \cite{Nak03}, \cite{Nak04}, \cite{Her06}. By the works \cite{HL10}, \cite{Nak11}, \cite{IIKKN13a}, \cite{IIKKN13b}, the T-systems are embedded in the cluster algebras associated with certain quivers.

In \cite{HL10}, Hernandez and Leclerc introduced the concept of monoidal categorifications of cluster algebras. They gave the following conjecture (we will recall the definition of $\mathcal{C}_{\ell}$ in Section \ref{section Hernandez-Leclerc conjecture}, see \cite{HL10}, \cite{HL13}).

\begin{conjecture}[{Conjecture 13.2 in \cite{HL10}, Conjecture 9.1 in \cite{Le10}, Conjecture 5.2 in \cite{HL13}}] \label{Hernandez-Leclerc conjecture 1}
The Grothendieck ring of some subcategory $\mathcal{C}_{\ell}$ of the category $\mathcal{C}$ of all finite-dimensional $U_q \widehat{\mathfrak{g}}$-modules has a cluster algebra structure. The simple $U_{q} \widehat{\mathfrak{g}}$-modules which are prime and real are cluster variables in some cluster algebra.
\end{conjecture}

In the case of types $A_n$ and $D_4$, $\ell = 1$, Conjecture \ref{Hernandez-Leclerc conjecture 1} is proved in \cite{HL10}. In the case of types $ADE$, $\ell=1$, Conjecture \ref{Hernandez-Leclerc conjecture 1} is proved in \cite{Nak11}. The work of \cite{Nak11} is generalized to all acyclic quivers by \cite{KQ14} and \cite{Lee13}. It is proved in \cite{HL13} that Conjecture \ref{Hernandez-Leclerc conjecture 1} is true for Kirillov-Reshetikhin modules in all types.

The realization of the T-systems for non-simply laced types, including type $B_n$, in cluster algebras was first given by \cite{IIKKN13a}, \cite{IIKKN13b}. In the paper \cite{HL13}, Hernandez and Leclerc apply the theory of cluster algebras to study the $q$-characters of a family of $U_q \widehat{\mathfrak{g}}$-modules called Kirillov-Reshetkhin modules and they give a new algorithm to compute the $q$-characters of these modules.

The family of minimal affinizations of quantum groups is an important family of simple modules of $U_q \widehat{\mathfrak{g}}$ which was introduced in \cite{C95}. The celebrated Kirillov-Reshetikhin modules are examples of minimal affinizations. Minimal affinizations are studied intensively in recent years, see for example, \cite{Her07}, \cite{MP07}, \cite{M10}, \cite{CG11}, \cite{MP11}, \cite{MY12a}, \cite{MY12b}, \cite{CMY13}, \cite{LM13}, \cite{Nao13}, \cite{Nao14}, \cite{MY14}, \cite{QL14}, \cite{S14}, \cite{Li15}.

Mukhin and Young introduced the extended T-systems in \cite{MY12b} and they showed that the extended T-systems of types $A_n$, $B_n$ are satisfied by the $q$-characters of a class of $U_q \widehat{\mathfrak{g}}$-modules called snake modules of types $A_n$, $B_n$. The class of snake modules contains all minimal affinizations.

In this paper, we use a new approach to study minimal affinizations. The family of minimal affinizations for $U_q \widehat{\mathfrak{g}}$ can be divided into two parts according to the highest weight monomials. For example, in type $A_n$, let (the notations will be explained in Section 2)
\begin{align*}
M^{(s)}_{k_{1},k_{2},\ldots,k_{n}} & =\prod_{j=1}^{n}\left(\prod_{i_{j}=0}^{k_{j}-1}j_{s+2\sum^{j-1}_{p=1}{k_p}+2i_{j}+(j-1)}\right),  \\
\widetilde{M}^{(s)}_{k_{1},k_{2},\ldots,k_{n}} & =\prod_{j=1}^{n}\left(\prod_{i_{j}=0}^{k_{j}-1}j_{-s-2\sum^{j-1}_{p=1}{k_p}-2i_{j}-(j-1)}\right).
\end{align*}
The first (resp. second) part of the family of minimal affinizations of type $A_n$ consists of minimal affinizations with highest weight monomials $M^{(s)}_{k_{1},k_{2},\ldots,k_{n}}$ (resp. $\widetilde{M}^{(s)}_{k_{1},k_{2},\ldots,k_{n}}$).
The M-systems (resp. dual M-systems) introduced in this paper are systems of equations which are satisfied by the $q$-characters of first (resp. second) part of the family of minimal affinizations of $U_q \widehat{\mathfrak{g}}$, Theorem \ref{M-systems} (resp. \ref{dual M-systems}).

The extended T-system of type $A_n$ (resp. $B_n$) is closed within the family of snake modules of type $A_n$ (resp. $B_n$), \cite{MY12b}. The M-system of type $A_{n}$ (resp. $B_{n}$) is closed within the family of minimal affinizations of type $A_{n}$ (resp. $B_{n}$). By using an elementary involution of the Grothendieck ring, we obtain the dual M-systems of types $A_{n}$ and $B_{n}$.

The union of the M-systems of types $A_n$, $B_n$, the dual M-systems of types $A_n$, $B_n$ (defined in Theorem \ref{dual M-systems}), and the T-systems of types $A_n$, $B_n$ is a closed system which contains all minimal affinizations of types $A_n$, $B_n$ (including Kirillov-Reshetikhin modules of types $A_n$, $B_n$).

The modules in the summands on the right hand side of each equation in the M-systems and dual M-systems are simple, Theorems \ref{irreducible}, \ref{dual M-systems}.

T-systems have many applications to mathematics and physics, see \cite{KNS11}. Since the equations in M-systems and dual M-systems have some nice properties and they are satisfied by $q$-characters of some family of $U_q \widehat{\mathfrak{g}}$-modules, we expect that M-systems and dual M-systems will have applications to mathematics and physics like T-systems.

We show that the equations in the M-system of type $A_{n}$ (resp. $B_{n}$) correspond to mutations in some cluster algebra $\mathscr{A}$ (resp. $\mathscr{A}'$) introduced in \cite{HL13}. Moreover, every minimal affinization in the M-system of type $A_n$ (resp. $B_n$) corresponds to a cluster variable in $\mathscr{A}$ (resp. $\mathscr{A}'$).

There are cluster algebras $\widetilde{\mathscr{A}}$, $\widetilde{\mathscr{A}'}$ which are dual to $\mathscr{A}$ and $\mathscr{A}'$ respectively such that every minimal affinization in the dual M-system of type $A_n$ (resp. $B_n$) corresponds to a cluster variable in $\widetilde{\mathscr{A}}$ (resp. $\widetilde{\mathscr{A}'}$).

We give a proof of the fact that minimal affinizations of types $A_n$ and $B_n$ are real. According to the results in \cite{CMY13}, minimal affinizations of all types are prime. Therefore minimal affinizations of type $A_n$ (resp. $B_n$) are simple, real, and prime and they correspond to cluster variables in $\mathscr{A}$ (resp. $\mathscr{A}'$). Thus we have shown that Conjecture \ref{Hernandez-Leclerc conjecture 1} is true for minimal affinizations of types $A_n$ and $B_n$.

We also have m-systems and dual m-systems of types $A_n, B_n$ which are obtained from M-systems and dual M-systems of types $A_n, B_n$ by restricting the modules in M-systems and dual M-systems to $U_q \mathfrak{g}$-modules, Sections \ref{m-systems Uqg} and \ref{dual m-systems Uqg}.

The M-systems also exist for other Dynkin types of minimal affinizations. The M-system of type $G_2$ is studied in the paper \cite{QL14}. Since the method of proving that the $q$-characters of minimal affinizations satisfy the M-systems of types $C$, $D$, $E$, $F$ are different from the method used in this paper, we will write them in other papers.

The paper is organized as follows. In Section 2, we give some background information about cluster algebras and representation theory of quantum affine algebras. In Section \ref{M-systems of types AB}, we describe the M-systems of types $A_n$, $B_n$. In Section \ref{relation between M-systems and cluster algebras}, we study relations between M-systems and cluster algebras. In Section \ref{dual M system section}, we study the dual M-systems of types $A_{n}, B_{n}$. In Section \ref{section Hernandez-Leclerc conjecture}, we show that the Hernandez-Leclerc conjecture is true for minimal affinizations of types $A_n$ and $B_n$. In Sections \ref{proof M-systems} and \ref{proof irreducible} we prove Theorems \ref{M-systems} and \ref{irreducible} given in Section \ref{M-systems of types AB}. In the Appendix, we give some examples of mutation sequences.

\section{Preliminaries}

\subsection{Cluster algebras}
Cluster algebras are invented by Fomin and Zelevinsky in \cite{FZ02}. Let $\mathbb{Q}$ be the rational field and $\mathcal{F}=\mathbb{Q}(x_{1}, x_{2}, \ldots, x_{n})$ the field of rational functions. A seed in $\mathcal{F}$ is a pair $\Sigma=({\bf y}, Q)$, where ${\bf y} = (y_{1}, y_{2}, \ldots, y_{n})$ is a free generating set of $\mathcal{F}$, and $Q$ is a quiver with vertices labeled by $1, 2, \ldots, n$. Assume that $Q$ has neither loops nor $2$-cycles. For $k=1, 2, \ldots, n$, one defines a mutation $\mu_k$ by $\mu_k({\bf y}, Q) = ({\bf y}', Q')$. Here ${\bf y}' = (y_1', \ldots, y_n')$, $y_{i}'=y_{i}$, for $i\neq k$, and
\begin{equation}
y_{k}'=\frac{\prod_{i\rightarrow k} y_{i}+\prod_{k\rightarrow j} y_{j}}{y_{k}}, \label{exchange relation}
\end{equation}
where the first (resp. second) product in the right hand side is over all arrows of $Q$ with target (resp. source) $k$, and $Q'$ is obtained from $Q$ by
\begin{enumerate}
\item[(i)] adding a new arrow $i\rightarrow j$ for every existing pair of arrow $i\rightarrow k$ and $k\rightarrow j$;

\item[(ii)] reversing the orientation of every arrow with target or source equal to $k$;

\item[(iii)] erasing every pair of opposite arrows possible created by (i).
\end{enumerate}
The mutation class $\mathcal{L}(\Sigma)$ is the set of all seeds obtained from $\Sigma$ by a finite sequence of mutations $\mu_{k}$. If $\Sigma'=((y_{1}', y_{2}', \ldots, y_{n}'), Q')$ is a seed in $\mathcal{L}(\Sigma)$, then the subset $\{y_{1}', y_{2}', \ldots, y_{n}'\}$ is called a $cluster$, and its elements are called \textit{cluster variables}. The \textit{cluster algebra} $\mathcal{A}_{\Sigma}$ is the subring of $\mathcal{F}$ generated by all cluster variables. \textit{Cluster monomials} are monomials in the cluster variables supported on a single cluster.

In this paper, the initial seed in the cluster algebra we use is of the form $\Sigma=({\bf y}, Q)$, where ${\bf y}$ is an infinite set and $Q$ is an infinite quiver.

\begin{definition}[{Definition 3.1, \cite{GG14}}] \label{definition of cluster algebras of infinite rank}
Let $Q$ be a quiver without loops or $2$-cycles and with a countably infinite number of vertices labelled by all integers $i \in \mathbb{Z}$. Furthermore, for each vertex $i$ of $Q$ let the
number of arrows incident with $i$ be finite. Let ${\bf y} = \{y_i \mid i \in \mathbb{Z}\}$. An infinite initial seed is a pair $({\bf y}, Q)$. By finite sequences of mutations
at vertices of $Q$ and simultaneous mutations of the set ${\bf y}$ using the exchange relation (\ref{exchange relation}), one obtains a family of infinite seeds. The sets of variables in these seeds are called
the infinite clusters and their elements are called the cluster variables. The cluster algebra of
infinite rank of type $Q$ is the subalgebra of $\mathbb{Q}({\bf y})$ generated by the cluster variables.
\end{definition}

\subsection{Quantum affine algebras} \label{definition of quantum affine algebras}
Let $\mathfrak{g}$ be a simple Lie algebra and $I=\{1, \ldots, n\}$ the indices of the Dynkin diagram of $\mathfrak{g}$ (we use the same labeling of the vertices of the Dynkin diagram of $\mathfrak{g}$ as the one used in \cite{Car05}). Let $C=(C_{ij})_{i,j\in I}$ be the Cartan matrix of $\mathfrak{g}$, where $C_{ij}=\frac{2 ( \alpha_i, \alpha_j ) }{( \alpha_i, \alpha_i )}$. There is a matrix $D=\diag(d_{i}\mid i\in I)$ with entries in $\mathbb{Z}_{>0}$ such that $B=DC=(b_{ij})_{i,j\in I}$ is symmetric. We have $D=\diag(d_{i}\mid i\in I)$, where $d_i=1$, $i\in I$, for type $A_n$ and $d_i=2$, $i=1,\ldots, n-1$, $d_n=1$, for type $B_n$. Let $t=\max \{d_{i}\mid i\in I\}$. Then $t=1$ for type $A_{n}$ and $t=2$ for type $B_{n}$.

Quantum groups are introduced independently by Jimbo \cite{Jim85} and Drinfeld \cite{Dri87}. Quantum affine algebras are infinite-dimensional quantum groups. In this paper, we take $q$ to be a non-zero complex number which is not a root of unity. The quantum affine algebra $U_q \widehat{\mathfrak{g}}$ in Drinfeld's new realization, see \cite{Dri88}, is generated by $x_{i, m}^{\pm}$ ($i\in I, m\in \mathbb{Z}$), $k_i^{\pm 1}$ $(i\in I)$, $h_{i, m}$ ($i\in I, m\in \mathbb{Z}\backslash \{0\}$) and central elements $c^{\pm 1/2}$, subject to certain relations.

The subalgebra of $U_q \widehat{\mathfrak{g}}$ generated by $(k_i^{\pm})_{i\in I}, (x_{i, 0}^{\pm})_{i\in I}$ is a Hopf subalgebra of $U_q \widehat{\mathfrak{g}}$ and is isomorphic as a Hopf algebra to the quantized enveloping algebra $U_q\mathfrak{g}$ of $\mathfrak{g}$. Therefore $U_q \widehat{\mathfrak{g}}$-modules restrict to $U_q\mathfrak{g}$-modules.

\subsection{Finite-dimensional $U_q \widehat{\mathfrak{g}}$-modules and $q$-characters}

In this section, we recall the standard facts about finite-dimensional $U_q \widehat{\mathfrak{g}}$-modules and their $q$-characters, see \cite{CP94}, \cite{CP95a}, \cite{FR98}.

A module $V$ of $U_q \widehat{\mathfrak{g}}$ is of type $1$ if $c^{\pm 1/2}$ acts as the identity on $V$ and
\begin{align} \label{decomposition}
V=\bigoplus_{\lambda \in P} V_{\lambda}, \ V_{\lambda} = \{v\in V  :   k_i v=q^{( \alpha_i, \lambda )} v \}.
\end{align}
In the following, all modules will be assumed to be finite-dimensional and of type $1$ without further comment. The decomposition (\ref{decomposition}) of a finite-dimensional module $V$ into its $U_q\mathfrak{g}$-weight spaces can be refined by decomposing it into the Jordan subspaces of the mutually commuting operators $\phi_{i, \pm r}^{\pm}$, see \cite{FR98}:
\begin{align}
V=\bigoplus_{\gamma} V_{\gamma}, \ \gamma=(\gamma_{i, \pm r}^{\pm})_{i\in I, r\in \mathbb{Z}_{\geq 0}}, \ \gamma_{i, \pm r}^{\pm} \in \mathbb{C},
\end{align}
where
\begin{align*}
V_{\gamma} = \{v\in V  :   \exists k\in \mathbb{N}, \forall i \in I, m \geq 0, (\phi_{i, \pm m}^{\pm} - \gamma_{i, \pm m}^{\pm})^{k} v =0 \}.
\end{align*}
Here $\phi_{i, n}^{\pm}$'s are determined by the formula
\begin{align}
\phi_i^{\pm}(u) = \sum_{n=0}^{\infty} \phi_{i, \pm n}^{\pm} u^{\pm n} = k_i^{\pm 1} \exp\left( \pm (q-q^{-1}) \sum_{m=1}^{\infty} h_{i, \pm m}u^{\pm m}\right).
\end{align}
If $\dim(V_{\gamma})>0$, then $\gamma$ is called an \textit{$l$-weight} of $V$. For every finite-dimensional $U_q \widehat{\mathfrak{g}}$-module, the $l$-weights are known, see \cite{FR98}, to be of the form
\begin{align}
\gamma_i^{\pm}(u) = \sum_{r=0}^{\infty} \gamma_{i, \pm r}^{\pm} u^{\pm r} = q_i^{\deg Q_i - \deg R_i} \frac{Q_i(uq_i^{-1})R_i(uq_i)}{Q_{i}(uq_i) R_{i}(uq_i^{-1})}, \label{gamma}
\end{align}
where the right hand side is to be treated as a formal series in positive (resp. negative) integer powers of $u$, and $Q_i, R_i$ are polynomials of the form
\begin{align}
Q_i(u) = \prod_{a\in \mathbb{C}^{\times}} (1-ua)^{w_{i, a}}, \ R_{i}(u)=\prod_{a\in \mathbb{C}^{\times}}(1-ua)^{x_{i, a}}, \label{QR}
\end{align}
for some $w_{i, a}, x_{i, a} \in \mathbb{Z}_{\geq 0}, i\in I, a\in \mathbb{C}^{\times}$. Let $\mathcal{P}$ denote the free abelian multiplicative group of monomials in infinitely many formal variables $(Y_{i, a})_{i\in I, a\in \mathbb{C}^{\times}}$.

There is a bijection $m$ from the set of $l$-weights of finite-dimensional modules to $\mathcal{P}$ given as follows. Let $\gamma$ be the $l$-weight given by (\ref{gamma}), (\ref{QR}). Then $m(\gamma)=\prod_{i\in I, a\in \mathbb{C}^{\times}} Y_{i, a}^{w_{i, a}-x_{i, a}}$. For an $l$-weight $\gamma$, we write $V_{m(\gamma)}$ for $V_{\gamma}$.

Let $\mathbb{Z}\mathcal{P} = \mathbb{Z}[Y_{i, a}^{\pm 1}]_{i\in I, a\in \mathbb{C}^{\times}}$ be the group ring of $\mathcal{P}$. The $q$-character of a $U_q \widehat{\mathfrak{g}}$-module $V$ is given by
\begin{align*}
\chi_q(V) = \sum_{m\in \mathcal{P}} \dim(V_{m}) m \in \mathbb{Z}\mathcal{P}.
\end{align*}
Let $\mathcal{R}$ be the Grothendieck ring of finite-dimensional $U_q \widehat{\mathfrak{g}}$-modules and $[V]\in \mathcal{R}$ the class of a finite-dimensional $U_q \widehat{\mathfrak{g}}$-module $V$. The $q$-character map defines an injective ring homomorphism, see \cite{FR98},
\begin{align*}
\chi_q: \mathcal{R} \to \mathbb{Z}\mathcal{P}.
\end{align*}

For any finite-dimensional $U_q \widehat{\mathfrak{g}}$-module $V$, denote by $\mathscr{M}(V)$ the set of all monomials in $\chi_q(V)$. For each $j\in I$, a monomial $m=\prod_{i\in I, a\in \mathbb{C}^{\times}} Y_{i, a}^{u_{i, a}}$, where $u_{i, a}$ are some integers, is said to be \textit{$j$-dominant} (resp. \textit{$j$-anti-dominant}) if and only if $u_{j, a} \geq 0$ (resp. $u_{j, a} \leq 0$) for all $a\in \mathbb{C}^{\times}$. A monomial is called \textit{dominant} (resp. \textit{anti-dominant}) if and only if it is $j$-dominant (resp. $j$-anti-dominant) for all $j\in I$. Let $\mathcal{P}_+ \subset \mathcal{P}$ denote the set of all dominant monomials and for $i \in I$, let $\mathcal{P}_{i,+} \subset \mathcal{P}$ denote the set of all $i$-dominant monomials.

Let $V$ be a $U_q \widehat{\mathfrak{g}}$-module and $m\in \mathscr{M}(V)$ a monomial. A non-zero vector $v\in V_m$ is called a \textit{highest $l$-weight vector} with \textit{highest $l$-weight} $\gamma(m)$ if
\begin{align*}
x_{i, r}^{+} \cdot v=0, \ \phi_{i, \pm t}^{\pm} \cdot v=\gamma(m)_{i, \pm t}^{\pm} v, \ \forall i\in I, r\in \mathbb{Z}, t\in \mathbb{Z}_{\geq 0}.
\end{align*}
The module $V$ is called a \textit{highest $l$-weight representation} if $V=U_q \widehat{\mathfrak{g}}\cdot v$ for some highest $l$-weight vector $v\in V$.

It is known, see \cite{CP94}, \cite{CP95a}, that for each $m_+\in \mathcal{P}_+$ there is a unique finite-dimensional irreducible module, denoted $L(m_+)$, of $U_q \widehat{\mathfrak{g}}$ that is highest $l$-weight with highest $l$-weight $\gamma(m_+)$, and moreover every finite-dimensional irreducible $U_q \widehat{\mathfrak{g}}$-module is of this form for some $m_+ \in \mathcal{P}_+$. We call $m_+$ the \textit{highest monomial} in $\chi_q(L(m_+))$. Also, if $m_+, m_+' \in \mathcal{P}_+$ and $m_{+}\neq m_{+}'$, then $L(m_+) \not\cong L(m_+')$.

For $b\in \mathbb{C}^{\times}$, define the shift of spectral parameter map $\tau_b: \mathbb{Z}\mathcal{P} \to \mathbb{Z}\mathcal{P}$ to be a homomorphism of rings sending $Y_{i, a}^{\pm 1}$ to $Y_{i, ab}^{\pm 1}$. Let $m_1, m_2 \in \mathcal{P}_+$. If $\tau_b(m_1) = m_2$, then $\tau_b \chi_q(L(m_1)) = \chi_q(L(m_2))$.

The concept of special module was introduced in \cite{Nak04} and the concept of anti-special module was introduced in \cite{Her07}. A finite-dimensional $U_q \widehat{\mathfrak{g}}$-module $V$ is said to be \textit{special} if and only if $\mathscr{M}(V)$ contains exactly one dominant monomial. It is called \textit{anti-special} if and only if $\mathscr{M}(V)$ contains exactly one anti-dominant monomial. It is called \textit{thin} if and only if all $l$-weight spaces of $V$ have dimensions less or equal to $1$. It is said to be \textit{prime} if and only if it is not isomorphic to a tensor product of two non-trivial $U_q \widehat{\mathfrak{g}}$-modules, see \cite{CP97}. Clearly, if a module is special or anti-special, then it is irreducible. A simple $U_q \widehat{\mathfrak{g}}$-modules $M$ is called real if $M\bigotimes M$ is simple, see \cite{Le03}.

For simplicity, we use $\chi_q(m_+)$ to denote $\chi_q(L(m_+))$ and use $\chi_{q}(m_1)\subseteq \chi_{q}(m_2)$ to denote $\mathscr{M}(L(m_1))\subseteq \mathscr{M}(L(m_2))$ for dominant monomials $m_+$, $m_1$, $m_2$.

The elements $A_{i, a} \in \mathcal{P}, i\in I, a\in \mathbb{C}^{\times}$, are defined by
\begin{align*}
A_{i, a} = Y_{i, aq_{i}}Y_{i, aq_{i}^{-1}} \prod_{C_{ji}=-1}Y_{j, a}^{-1} \prod_{C_{ji}=-2}Y_{j, aq}^{-1}Y_{j, aq^{-1}}^{-1}\prod_{C_{ji}=-3}Y_{j, aq^{2}}^{-1}Y_{j, a}^{-1}Y_{j, aq^{-2}}^{-1},
\end{align*}
where $q_i = q^{d_i}$ ($d_i$ is defined in Section \ref{definition of quantum affine algebras}), see \cite{FR98}.

Let $\mathcal{Q}$ be the subgroup of $\mathcal{P}$ generated by $A_{i, a}, i\in I, a\in \mathbb{C}^{\times}$. Let $\mathcal{Q}^{\pm}$ be the monoids generated by $A_{i, a}^{\pm 1}, i\in I, a\in \mathbb{C}^{\times}$. There is a partial order $\leq$ on $\mathcal{P}$ in which
\begin{align}
m\leq m' \text{ if and only if } m'm^{-1}\in \mathcal{Q}^{+}. \label{partial order of monomials}
\end{align}
For all $m_+ \in \mathcal{P}_+$, $\mathscr{M}(L(m_+)) \subset m_+\mathcal{Q}^{-}$, see \cite{FM01}.

We will need the concept \text{right-negative} to classify dominant monomials. Let $m$ be a monomial. If for all $a \in \mathbb{C}^{\times}$ and $i\in I$, we have the property: if the power of $Y_{i, a}$ in $m$ is non-zero and the power of $Y_{j, aq^k}$ in $m$ is zero for all $j \in I, k\in \mathbb{Z}_{>0}$, then the power of $Y_{i, a}$ in $m$ is negative, then the monomial $m$ is called \textit{right-negative}, see \cite{FM01}. For $i\in I, a\in \mathbb{C}^{\times}$, $A_{i,a}^{-1}$ is right-negative. A product of right-negative monomials is right-negative. If $m$ is right-negative and $m'\leq m$, then $m'$ is right-negative, see \cite{FM01}, \cite{Her06}.

\subsection{Minimal affinizations of $U_q\mathfrak{g}$-modules}
Minimal affinizations of $U_q\mathfrak{g}$-modules are $U_q \widehat{\mathfrak{g}} $-modules introduced by Chari in \cite{C95}. They are defined as follows.

Let $\omega: \mathcal{P} \to P^+$ be a map defined by for $m = \prod_{i \in I, a \in \mathbb{C}^{\times}} Y_{i, a}^{u_{i,a}}$, $\omega(m) = \sum_{i\in I, a \in \mathbb{C}^{\times}} u_{i,a} \omega_i$.
\begin{definition}[{\cite{C95}}] Let $V$ be a simple finite-dimensional $U_q\mathfrak{g}$-module, a simple finite-dimensional $U_q \widehat{\mathfrak{g}} $-module $L(m)$ is said to be an affinization of $V$ if $\omega(m)$ is the highest weight of $V$.
\end{definition}

Let $V$ be a finite-dimensional $U_q\mathfrak{g}$-module and $\lambda \in P$. The multiplicity of the simple $U_q\mathfrak{g}$-module of highest weight $\lambda$ in $V$ is denoted by $m_{\lambda}(V)$.  Two affinizations are said to be equivalent if they are isomorphic as $U_q\mathfrak{g}$-modules. Let $\mathcal{Q}_V$ be the equivalence classes of affinizations of $V$. The equivalence class of an affinization $L$ of $V$ is denoted by $[L] \in \mathcal{Q}_V$. Let $[L], [L'] \in \mathcal{Q}_V$. We write $[L] \leq [L']$ if and only if for all $ \mu \in P $, either $m_{\mu}(L) \leq m_{\mu}(L')$ or there exists $\nu > \mu$ such that $m_{\nu}(L) < m_{\mu}(L')$. Then ``$\leq $" defines a partial ordering on $\mathcal{Q}_V$.

\begin{definition}[{\cite{C95}}]
A minimal affinizations of $V$ is a minimal element of $\mathcal{Q}_V$ with respect to the partial ordering ``$\leq $".
\end{definition}

In the case of $\mathfrak{g} = \mathfrak{sl}_2$, the minimal affinizations are evaluation modules.

From now on, we fix an $a\in \mathbb{C}^{\times}$ and denote $i_s = Y_{i, aq^s}$, $i \in I$, $s \in \mathbb{Z}$. Let $\lambda = k_1 \omega_1 + \cdots k_n \omega_n$, $k_1, \ldots, k_n \in \mathbb{Z}_{\geq 0}$, and let $V(\lambda)$ be the simple $U_q \mathfrak{g}$-module with highest weight $\lambda$. Without loss of generality, we may assume that in type $A_n$ a simple $U_q \widehat{\mathfrak{g}} $-module $L(m_+)$ is a minimal affinization of $V(\lambda)$ if and only if $m_+$ is one of the following monomials:
\begin{align}
M^{(s)}_{k_{1},k_{2},\ldots,k_{n}} & =\prod_{j=1}^{n}\left(\prod_{i_{j}=0}^{k_{j}-1}j_{s+2\sum^{j-1}_{p=1}{k_p}+2i_{j}+(j-1)}\right), \label{minimal affinizations An I} \\
\widetilde{M}^{(s)}_{k_{1},k_{2},\ldots,k_{n}} & =\prod_{j=1}^{n}\left(\prod_{i_{j}=0}^{k_{j}-1}j_{-s-2\sum^{j-1}_{p=1}{k_p}-2i_{j}-(j-1)}\right), \label{minimal affinizations An II}
\end{align}
where $s\in\mathbb{Z}$, see \cite{CP96a}. Similarly, we may assume that in type $B_n$, a simple $U_q \widehat{\mathfrak{g}} $-module $L(m_+)$ is a minimal affinization of $V(\lambda)$ if and only if $m_+$ is one of the following monomials:
\begin{align}
M_{k_{1},k_{2},\ldots,k_{n}}^{(s)} & =\left(\prod_{j=1}^{n-1}
\left(\prod_{i_{j}=0}^{k_{j}-1} j_{s+4\sum_{p=1}^{j-1}k_{p}+4i_{j}+2j-2}\right)\right)\prod_{i_{n}=0}^{k_{n}-1}n_{s+4\sum_{p=1}^{n-1}k_{p}+2i_{n}+2n-3}, \label{minimal affinizations Bn I}  \\
\widetilde{M}_{k_{1},k_{2},\ldots,k_{n}}^{(s)} & =\left(\prod_{j=1}^{n-1}
\left(\prod_{i_{j}=0}^{k_{j}-1} j_{-s-4\sum_{p=1}^{j-1}k_{p}-4i_{j}-2j+2}\right)\right)\prod_{i_{n}=0}^{k_{n}-1}n_{-s-4\sum_{p=1}^{n-1}k_{p}-2i_{n}-2n+3}, \label{minimal affinizations Bn II}
\end{align}
where $s\in\mathbb{Z}$, see \cite{CP95b}. We denote $A_{i,aq^{s}}^{-1}$ by $A_{i,s}^{-1}$. We use $\mathcal{M}^{(s)}_{k_{1},k_{2},\ldots,k_{n}}$ (resp. $\widetilde{\mathcal{M}}_{k_{1},k_{2},\ldots,k_{n}}^{(s)})$ to denote the irreducible finite-dimensional $U_{q}\widehat{\mathfrak{g}}$-module with highest $l$-weight $M^{(s)}_{k_{1},k_{2},\ldots,k_{n}}$ (resp. $\widetilde{M}_{k_{1},k_{2},\ldots,k_{n}}^{(s)})$, where $k_{1},\ldots k_{n}\in  \mathbb{Z}_{\geq 0}$.

\begin{theorem}[Theorem 3.8, {\cite{Her07}}] \label{dual modules are special and antispecial}
In the case of type $A_{n}$ (resp. type $B_{n}$), the modules $\mathcal{M}_{k_{1},\ldots,k_{n}}^{(s)}$ and $\widetilde{\mathcal{M}}_{k_{1},\ldots,k_{n}}^{(s)}$, $s\in \mathbb{Z}$, $k_{1},\ldots, k_{n}\in  \mathbb{Z}_{\geq 0}$ are special and anti-special.
\end{theorem}

\subsection{$q$-characters of $U_q \widehat{\mathfrak{sl}}_2$-modules and the Frenkel-Mukhin algorithm}

The $q$-characters of $U_q \widehat{\mathfrak{sl}}_2$-modules are well-understood, see \cite{CP91}, \cite{FR98}. We recall some results which will be used in this paper.

Let $W_{k}^{(a)}$ be the irreducible representation $U_q \widehat{\mathfrak{sl}}_2$ with
highest weight monomial
\begin{align}
X_{k}^{(a)}=\prod_{i=0}^{k-1} Y_{aq^{k-2i-1}},
\end{align}
where $Y_a=Y_{1, a}$. Then the $q$-character of $W_{k}^{(a)}$ is given by
\begin{align} \label{q-characters of Uqsl_2 KR module}
\chi_q(W_{k}^{(a)})=X_{k}^{(a)} \sum_{i=0}^{k} \prod_{j=0}^{i-1} A_{aq^{k-2j}}^{-1},
\end{align}
where $A_a=Y_{aq^{-1}}Y_{aq}$.

For $a\in \mathbb{C}^{\times}, k\in \mathbb{Z}_{\geq 1}$, the set $\Sigma_{k}^{(a)}=\{aq^{k-2i-1}\}_{i=0, \ldots, k-1}$ is called a \textit{string}. Two strings $\Sigma_{k}^{(a)}$ and $\Sigma_{k'}^{(a')}$ are said to be in \textit{general position} if the union $\Sigma_{k}^{(a)} \cup \Sigma_{k'}^{(a')}$ is not a string or $\Sigma_{k}^{(a)} \subset \Sigma_{k'}^{(a')}$ or $\Sigma_{k'}^{(a')} \subset \Sigma_{k}^{(a)}$.

Denote by $L(m_+)$ the irreducible $U_q \widehat{\mathfrak{sl}}_2$-module with
highest weight monomial $m_+$. Let $m_{+} \neq 1$ and $m_+\in \mathbb{Z}[Y_a]_{a\in \mathbb{C}^{\times}}$ be a dominant monomial. Then $m_+$ can be uniquely (up to permutation) written in the form
\begin{align*}
m_+=\prod_{i=1}^{s} \left( \prod_{b\in \Sigma_{k_i}^{(a_i)}} Y_{b} \right),
\end{align*}
where $s$ is an integer, $\Sigma_{k_i}^{(a_i)}, i=1, \ldots, s$, are strings which are pairwise in general position and
\begin{align} \label{q-characters of Uqsl_2 module}
L(m_+)=\bigotimes_{i=1}^s W_{k_i}^{(a_i)}, \qquad \chi_q(L(m_+))=\prod_{i=1}^s
 \chi_q(W_{k_i}^{(a_i)}).
\end{align}

For $j\in I$, let
\begin{align*}
\beta_j : \mathbb{Z}[Y_{i, a}^{\pm 1}]_{i\in I; a\in \mathbb{C}^{\times}} \to \mathbb{Z}[Y_{a}^{\pm 1}]_{a\in \mathbb{C}^{\times}}
\end{align*}
be the ring homomorphism such that for all $a\in \mathbb{C}^{\times}$, $Y_{k, a} \mapsto 1$ for $k\neq j$ and $Y_{j, a} \mapsto Y_{a}$.

Let $V$ be a $U_q \widehat{\mathfrak{g}}$-module. Then $\beta_i(\chi_q(V))$, $i\in \{1, 2, \ldots, n\}$, is the $q$-character of $V$ considered as a $U_{q_i}\widehat{\mathfrak{sl}}_2$-module.

In some situation, we can use the $q$-characters of $U_q \widehat{\mathfrak{sl}}_2$-modules to compute the $q$-characters of $U_q \widehat{\mathfrak{g}}$-modules for arbitrary $\mathfrak{g}$,  see Section 5 in \cite{FM01}. The corresponding algorithm is called the Frenkel-Mukhin algorithm. The Frenkel-Mukhin algorithm recursively computes the minimal possible $q$-character which contains $m_+$ and is consistent when restricted to $U_{q_i}\widehat{\mathfrak{sl}}_2$, $i\in \{1, 2, \ldots, n\}$.

In some cases, the Frenkel-Mukhin algorithm doesn't return all terms in the $q$-character of a module. There are some counter-examples given in \cite{NN11}. However, Frenkel-Mukhin algorithm produces the correct $q$-characters of modules in many cases. In particular, if a module $L(m_+)$ is special, then the Frenkel-Mukhin algorithm applied to $m_+$, see \cite{FM01}, produces the correct $q$-character $\chi_q(L(m_+))$.

We will need the following proposition from \cite{HL10}.

\begin{proposition}[{Proposition 5.9 in \cite{HL10}, Proposition 3.1 in \cite{Her05}}] \label{decomposition of a q-character}
Let $V$ be a $U_q \hat{\mathfrak{g}}$-module and fix $i \in I$. Then there is a unique
decomposition of $\chi_q(V)$ as a finite sum
\begin{align}
\chi_q(V) = \sum_{m \in \mathcal{P}_{i, +}} \lambda_m \varphi_i(m),
\end{align}
and the $\lambda_m$ are non-negative integers.
\end{proposition}
Here $\varphi_i(m)$ ($m \in \mathcal{P}_{i,+}$) is a polynomial defined as follows, see Section 5.2.1 of \cite{HL10}. Let $m \in \mathcal{P}_{i,+}$ be an $i$-dominant monomial. Let $\overline{m}$ be the monomial obtained from $m$ by replacing $Y_{j,a}$ by $Y_a$ if $j = i$ and by $1$ if $j \neq i$. Then the $q$-character $\chi_q(L(\overline{m}))$ of the $U_q \widehat{\mathfrak{sl}}_2$-module $L(\overline{m})$ is given by (\ref{q-characters of Uqsl_2 KR module}), (\ref{q-characters of Uqsl_2 module}). Write $\chi_q(L(\overline{m})) = \overline{m}(1 + \sum_{p} \overline{M}_p)$, where the $\overline{M}_p$ are monomials in the variables $A^{-1}_{a}$ $(a \in \mathbb{C}^{\times})$. Then one sets $\varphi_i(m) := m(1 + \sum_p M_p)$ where each $M_p$ is obtained from the corresponding $\overline{M}_p$ by replacing each variable $A^{-1}_{a}$ by $A^{-1}_{i,a}$.

The following corollary follows from Proposition \ref{decomposition of a q-character}, see \cite{HL10}.
\begin{corollary}[{\cite{HL10}}] \label{Uqsl_2 arguments}
Let $m \in \mathcal{P}_{+}$ and let $mM$ be a monomial of $\chi_q(L(m))$, where $M$ is a monomial in the $A_{j,a}^{-1}$ ($j \in I$). If $M$ contains no variable $A_{i,a}$, then
$mM \in \mathcal{P}_{i,+}$ and $\varphi_i(mM)$ is contained in $\chi_q(L(m))$. In particular, $\varphi_i(m)$ is contained in $\chi_q(L(m))$.
\end{corollary}

\subsection{Path description of $q$-characters of types $A_{n}$, $B_{n}$}
We will need the path description of $q$-characters of minimal affinizations of types $A_{n}$, $B_{n}$ which are introduced in Section 5 of \cite{MY12a}, see also Section 3 and Section 6 of \cite{MY12b} to classify dominant monomials in tensor products.

The explicit tableaux formulas of $q$-characters of minimal affinizations of types $A_{n}$, $B_{n}$ are given in \cite{Her07}.

The \textit{length} of $\mathcal{M}^{(s)}_{k_{1},k_{2},\ldots,k_{n}}$ is defined as $k_1+k_2 + \cdots +k_n$.
\begin{theorem}[{Theorem 6.1, \cite{MY12a}}] \label{path description of q-characters}  Suppose that the length of the minimal affinization $\mathcal M^{(s)}_{k_{1},k_{2},\ldots,k_{n}}$ is $L$. Then in the cases of types $A_n$ and $B_n$, we have
\begin{align}
\chi_{q}(\mathcal{M}^{(s)}_{k_{1},k_{2},\ldots,k_{n}})=\sum_{(p_{1},\ldots,p_{L})\in \overline{\varphi}_{(i_{t},k_{t})_{1\leq t\leq L}}}\prod_{t=1}^{L}m(p_{t}).
\end{align}
\end{theorem}
Now we explain the notations in Theorem \ref{path description of q-characters}, see \cite{MY12a}. A path is a finite sequence of points in the plane $\mathbb{R}^{2}$. In the case of type $A_{n}$, let
\begin{align*}
\mathcal{D}=\{(i,k)\in I\times \mathbb{Z}: i-k \equiv 1 \pmod 2\}.
\end{align*}
For all $(i,k)\in \mathcal{D}$, let
\begin{align*}
\mathfrak{\varphi}_{i,k}=\{ & ((0,y_{0}),(1,y_{1}),\ldots,(n+1,y_{n+1})):\\
& y_{0}=i+k, \ y_{n+1}=n+1-i+k, \text{ and } y_{i+1}-y_{i}\in \{1,-1\}, \  0\leq i\leq n\}.
\end{align*}
The sets $C_{p}^{\pm}$ of upper and lower corners of a path $p=((r,y_{r}))_{0\leq r \leq n+1}\in \mathfrak{\varphi}_{i,k}$ are defined as follows:
\begin{align*}
C^{+}_{p}=\{(r,y_{r})\in p: r\in I, \ y_{r-1}=y_{r}+1=y_{r+1}\},\\
C^{-}_{p}=\{(r,y_{r})\in p: r\in I, \ y_{r-1}=y_{r}-1=y_{r+1}\}.
\end{align*}In the case of type $B_{n}$, let
\begin{align*}
\mathcal{D}=\{(n,2k+1): k\in \mathbb{Z}\}\sqcup \{(i,k)\in I\times \mathbb{Z}: i<n \text{ and } k\equiv0 \pmod 2\}.
\end{align*}
It is written that $(j,l)\in p$ if $(j,l)$ is a point of the path $p$.

Fix an $\varepsilon$, $0<\varepsilon <1/2$, $\mathfrak{\varphi}_{n,l}$ for all $l\in 2\mathbb{Z}+1$ are defined as follows.
For all $l\equiv 3$ mod $4$,
\begin{gather}
\begin{align*}
\mathfrak{\varphi}_{n,l}= \{ & ((0,y_{0}), (2,y_{1}), \ldots, (2n-4,y_{n-2}), (2n-2,y_{n-1}),(2n-1,y_{n})): \\
 & y_{0}=l+2n-1, y_{i+1}-y_{i}\in \{2,-2\}, \  0\leq i\leq n-2, \text{ and } y_{n}-y_{n-1}\in \{1+\epsilon, -1-\epsilon\}\}.
\end{align*}
\end{gather}
For all $l\equiv 1$ mod $4$,
\begin{gather}
\begin{align*}
\mathfrak{\varphi}_{n,l}=& \{((4n-2,y_{0}),(4n-4,y_{1}),\ldots,(2n+2,y_{n-2}), (2n,y_{n-1}),(2n-1,y_{n})): \\
& y_{0}=l+2n-1, \ y_{i+1}-y_{i}\in \{2,-2\}, \   0\leq i\leq n-2,  \text { and } y_{n}-y_{n-1}\in \{1+\epsilon, -1-\epsilon\}\}.
\end{align*}
\end{gather}
For all $(i,k)\in \mathcal{D}$, $i< n$, $\mathfrak{\varphi}_{i,k}$ are defined as follows:
\begin{align*}
\mathfrak{\varphi}_{i,k}=&\{(a_{0},a_{1},\ldots,a_{n},\overline{a}_{n},\ldots,\overline{a}_{1},\overline{a}_{0}):  (a_{0},a_{1},\ldots,a_{N})\in \mathfrak{\varphi}_{n,k-(2n-2i-1)},\\
&  \quad(\overline{a}_{0},\overline{a}_{1},\ldots,\overline{a}_{n})\in \mathfrak{\varphi}_{n,k+(2n-2i-1)},
 \text { and } a_{n}-\overline{a}_{n}= (0,y) \text{ where } y>0\}.
\end{align*}
For all $(i,k)\in \mathcal{D}$, the sets of upper and lower corners $C^{\pm}_{p}$ of a path
\begin{align*}
p=((j_{r},l_{r}))_{0\leq r\leq |p|-1}\in \mathfrak{\varphi}_{i,k},
\end{align*}
where $|p|$ is the number of points in the path $p$, are defined as follows:
\begin{align*}
C^{+}_{p}=&\tau^{-1}\{(j_{r},l_{r})\in p: j_{r}\not \in \{0,2n-1,4n-2\}, \ l_{r-1}>l_{r},l_{r+1}>l_{r}\} \\
& \quad \bigsqcup \{(n,l)\in \mathcal{D}: (2n-1,l-\epsilon)\in p \text{ and }(2n-1,l+\epsilon)\not\in p\},
\end{align*}
\begin{align*}
C^{-}_{p}=&\tau^{-1}\{(j_{r},l_{r})\in p: j_{r}\not \in \{0,2n-1,4n-2\}, \ l_{r-1}<l_{r},l_{r+1}<l_{r}\} \\
& \quad \bigsqcup \{(n,l)\in \mathcal{D}: (2n-1,l-\epsilon)\not\in p \text{ and } (2n-1,l+\epsilon)\in p\},
\end{align*}
where $\tau$ is defined as follows:
\begin{align*}
\tau(i,k) =
\begin{cases}
     (2i, k),  & \text{if } i<n \text{ and } 2n+k-2i\equiv2 \pmod 4,  \\
     (4n-2-2i,k), & \text{if } i<n \text{ and } 2n+k-2i\equiv0 \pmod 4, \\
     (2n-1, k), & \text{if } i=n.
\end{cases}
\end{align*}
A map $m$ sending paths to monomials is defined by
\begin{align} \label{map sending paths to monomials}
m: \bigsqcup_{(i,k)\in \mathcal{D}}\varphi_{i,k} \to \mathbb{Z}[Y^{\pm}_{j,l}]_{(j,l)\in \mathcal{D}}; \quad  p \mapsto m(p)=\prod_{(j,l)\in C^{+}_{p}}Y_{j,l}\prod_{(j,l)\in C^{-}_{p}}Y^{-1}_{j,l}.
\end{align}
We identify a path $p$ with the monomial $m(p)$.

Let $p,p'$ be paths. It is said that $p$ is \textit{strictly above} $p'$ or $p'$ is \textit{strictly below} $p$ if
\begin{align*}
(x,y)\in p \text{ and } (x,z)\in p' \Longrightarrow y< z.
\end{align*}

It is said that a $T$-tuple of paths $(p_{1},\ldots,p_{T})$ is \textit{non-overlapping} if $p_{s}$ is strictly above $p_{t}$ for all $s<t$.

For any $(i_{t},k_{t})\in \mathcal{D}$, $1\leq t\leq T$, $T\in \mathbb{Z}_{\geq1}$, $\overline{\varphi}_{(i_{t},k_{t})_{1\leq t\leq T}}$ is defined by
\begin{align*}
\overline{\varphi}_{(i_{t},k_{t})_{1\leq t\leq T}}=\{(p_{1},\ldots,p_{T}): p_{t}\in \varphi_{i_{t},k_{t}}, \ 1\leq t\leq T, \ (p_{1},\ldots,p_{T})\text { is } \text {non-overlapping} \}.
\end{align*}
By Theorem \ref{path description of q-characters}, the $q$-character of a minimal affinization $\mathcal{M}^{(s)}_{k_{1},k_{2},\ldots,k_{n}}$ of type $A_n$ or $B_n$ with length $M$ is given by a set of $M$-tuples of non-overlapping paths. The paths in each $M$-tuple in this set are non-overlapping, this property is called \textit{non-overlapping property}.

We also need the following notations in this paper. For all $(i,k)\in \mathcal{D}$, let $p^{+}_{i,k}$ be the highest path which is the unique path in $\varphi_{i,k}$ with no lower corners and $p^{-}_{i,k}$ the lowest path which is the unique path in $\varphi_{i,k}$ with no upper corners.

\section{M-systems of types $A_{n}$, $B_{n}$}\label{M-systems of types AB}

In this section, we describe M-systems of types $A_{n}$, $B_{n}$.

\subsection{Neighbouring points}\label{Section neighbouring points}
We recall the concept of \textit{neighbouring points} which is introduced in Section $3$ of \cite{MY12b}.
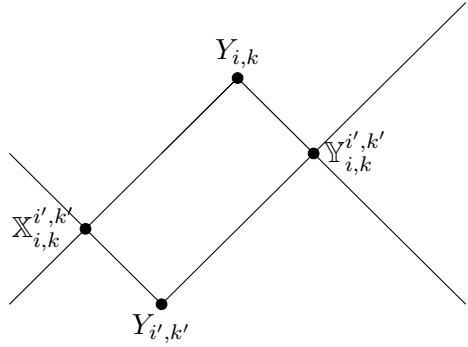
\begin{figure}[H]
\centerline{
\begin{tikzpicture}
\draw[fill] (1,1)--(2,2)circle (2pt)--(3,3)--(4,4) circle (2pt)--(5,3)circle (2pt)--(6,2)--(7,1);
\draw[fill] (1,3)--(2,2)--(3,1)circle (2pt)--(4,2)--(5,3)--(6,4)--(7,5);
\node [above] at (4,4) {$Y_{i,k}$};
\node [below] at (3,1) {$Y_{i',k'}$};
\node [left]  at (2,2) {$\mathbb{X}_{i,k}^{i',k'}$};
\node [right] at (5,3) {$\mathbb{Y}_{i,k}^{i',k'}$};
\end{tikzpicture}}
\end{figure}

In type $A_n$, let
\begin{align*}
\mathbb{X}_{i,k}^{i',k'}=
\begin{cases}
((\frac{1}{2}(i+k+i'-k'),\frac{1}{2}(i+k-i'+k'))), & k+i > k'-i',\\
\emptyset, & k+i = k'-i',
\end{cases}
\end{align*}
\begin{align*}
\mathbb{Y}_{i,k}^{i',k'}=
\begin{cases}
((\frac{1}{2}(i'+k'+i-k),\frac{1}{2}(i'+k'-i+k))), & k+n+1-i > k'-n-1+i',\\
\emptyset, & k+n+1-i = k'-n-1+i'.
\end{cases}
\end{align*}

In type $B_n$, let
\begin{gather}
\begin{align*}
(\mathbb{X}_{i,k}^{i',k'},\mathbb{Y}_{i,k}^{i',k'})=
\begin{cases}
(B_{i,k}^{i',k'},F_{i,k}^{i',k'}), & i<n,\ 2n+k-2i \equiv 1 \pmod 4, \text{ or } i=n,\ k \equiv 0 \pmod 4,\\
(F_{i,k}^{i',k'},B_{i,k}^{i',k'}), & i<n,\ 2n+k-2i \equiv 3 \pmod 4, \text{ or } i=n,\ k \equiv 2 \pmod 4,
\end{cases}
\end{align*}
\end{gather}
where
\begin{gather}
\begin{align*}
&B_{i,k}^{i',k'}=
\begin{cases}
\emptyset, & i<n,\ i'<n,\ k'-k=2i+2i',\\
((\frac{1}{4}(2i+k+2i'-k'),\frac{1}{2}(2i+k-2i'+k'))), & i<n,\ i'<n,\ k'-k<2i+2i',\\
\emptyset, & i<n,\ i'=n,\ k'-k=2i+2n-1,\\
((\frac{1}{4}(2i+k+2n-1-k'),\frac{1}{2}(2i+k-2n+1+k'))), & i<n,\ i'=n,\ k'-k<2i+2n-1,\\
((n,k'-2n+1+2i')), & i=n,\ i'<n,\\
\emptyset, & i=n,\ i'=n,
\end{cases}\\
&F_{i,k}^{i',k'}=
\begin{cases}
((\frac{1}{4}(2i'+k'+2i-k),\frac{1}{2}(2i'+k'-2i+k))), & i<n,\ i'<n,\ k'-k\ \leq 4n-4-2i-2i',\\
((n,k+2n-1-2i),(n,k'-2n+1+2i')), & i<n,\ i'<n,\ k'-k\ \geq 4n-2i-2i',\\
((n,k+2n-1-2i)), & i<n,\ i'=n,\\
((\frac{1}{4}(2n-1+k+2i'-k'),\frac{1}{2}(2n-1+k-2i'+k'))), & i=n,\ i'<n,\ k'-k<2n-1-2i',\\
\emptyset, & i=n,\ i'<n,\ k'-k=2n-1-2i',\\
((\frac{1}{4}(4n-2+k-k'), \frac{1}{2}(k+k'))), & i=n,\ i'=n.
\end{cases}
\end{align*}
\end{gather}

\subsection{M-systems of types $A_n$ and $B_n$} When we write $\mathcal{M}^{(s-2)}_{0,\ldots,0,\underset{i}{k_{i}},0, \ldots, 0,\underset{j}{k_{j}+1},k_{j+1},\ldots,k_{n}}$, we mean ${k_{i}}$ is in the $i$-th position,  $k_{j}+1$ is in the $j$-th position.
Our first main result in this paper is the following systems which we call M-systems of types $A_n$ and $B_{n}$ respectively.

In the case of type $A_n$, every minimal affinization which is not a Kirillov-Reshetikhin module can be written as
\begin{align}\label{M_2 in type A_n}
\mathcal{M}^{(s-2)}_{0,\ldots, 0,\underset{i}{k_{i}},0,\ldots,0,\underset{j}{k_{j}+1},k_{j+1},\ldots,k_{n}},
\end{align}
where $1 \leq i \leq n-1$, $i < j \leq n$, $k_i \geq 1$, $s\in \mathbb{Z}$, $k_{j}, \ldots, k_{n} \in \mathbb{Z}_{\geq 0}$.

In type $B_n$, every minimal affinization which is not a Kirillov-Reshetikhin module is one of the following modules
\begin{align}
& \mathcal{M}^{(s-4)}_{0,\ldots,0,\underset{i}{k_{i}},0,\ldots,0,\underset{j}{k_{j}+1},k_{j+1},\ldots,k_{n}}, \quad 1 \leq i \leq n-2, \ i < j \leq n-1, \label{M_2 in type B_n 1} \\
& \mathcal{M}^{(s-4)}_{0,\ldots,0,\underset{i}{k_{i}},0,\ldots,0,\underset{n}{k_n+2}}, \quad 1 \leq i \leq n-1,  \label{M_2 in type B_n 2} \\
& \mathcal{M}^{(s-4)}_{0,\ldots,0,\underset{i}{k_{i}},0,\ldots,0,\underset{n}{1}}, \quad 1 \leq i \leq n-1, \label{M_2 in type B_n 3}
\end{align}
where $k_i \geq 1$, $s\in \mathbb{Z}$, $k_{j}, \ldots, k_{n} \in \mathbb{Z}_{\geq 0}$.

In type $A_n$, let $\mathcal{M}_2$ be the minimal affinization (\ref{M_2 in type A_n}). In type $B_n$, let $\mathcal{M}_2$ be one of the minimal affinizations (\ref{M_2 in type B_n 1}), (\ref{M_2 in type B_n 2}), (\ref{M_2 in type B_n 3}).

In type $A_n$, let
\begin{equation}\label{M_1 in type A_n}
\begin{split}
\mathcal{M}_1=\mathcal{M}^{(s)}_{0,\ldots, 0,\underset{i}{k_{i}},0,\ldots,0, \underset{j}{k_{j}}, k_{j+1},\ldots,k_{n}}.
\end{split}
\end{equation}

In type $B_n$, let
\begin{equation}\label{M_1 in type B_n}
\mathcal{M}_1=
\begin{cases}
\mathcal{M}^{(s)}_{0,\ldots,0,\underset{i}{k_{i}},0,\ldots,0,\underset{j}{k_{j}},k_{j+1},\ldots,k_{n}}, & \text{ if $\mathcal{M}_2$ is (\ref{M_2 in type B_n 1})},\\
\mathcal{M}^{(s)}_{0,\ldots,0,\underset{i}{k_{i}},0,\ldots,0,\underset{n}{k_n}}, & \text{ if $\mathcal{M}_2$ is (\ref{M_2 in type B_n 2})},\\
\mathcal{M}^{(s)}_{0,\ldots,0,\underset{i}{k_{i}},0,\ldots,0,\underset{n}{0}}, & \text{ if $\mathcal{M}_2$ is (\ref{M_2 in type B_n 3})}.
\end{cases}
\end{equation}

In type $A_n$, let
\begin{align}\label{M_3 M_4 in type A_n}
\mathcal{M}_3 = \mathcal{M}^{(s-2)}_{0,\ldots, 0,\underset{i}{k_i+1},0,\ldots,0, \underset{j}{k_j}, k_{j+1},\ldots,k_{n}}, \quad
\mathcal{M}_4 = \mathcal{M}^{(s)}_{0,\ldots, 0,\underset{i}{k_i-1},0,\ldots,0,\underset{j}{k_j+1},k_{j+1},\ldots,k_{n}}.
\end{align}

In type $B_n$, let
\begin{equation}\label{M_3 M_4 in type B_n}
\begin{split}
&\mathcal{M}_3 = \begin{cases}
\mathcal{M}^{(s-4)}_{0,\ldots,0,\underset{i}{k_i+1},0,\ldots,0,\underset{j}{k_j},k_{j+1},\ldots,k_{n}}, & \text{ if $\mathcal{M}_2$ is (\ref{M_2 in type B_n 1})},\\
\mathcal{M}^{(s-4)}_{0,\ldots,0,\underset{i}{k_i+1},0,\ldots,0,\underset{n}{k_n}}, & \text{ if $\mathcal{M}_2$ is (\ref{M_2 in type B_n 2})},\\
\mathcal{M}^{(s-4)}_{0,\ldots,0,\underset{i}{k_i+1},0,\ldots,0,\underset{n}{0}}, & \text{ if $\mathcal{M}_2$ is (\ref{M_2 in type B_n 3})},
\end{cases} \\
&\mathcal{M}_4 =
\begin{cases}
\mathcal{M}^{(s)}_{0,\ldots,0,\underset{i}{k_i-1},0,\ldots,0,\underset{j}{k_{j}+1},k_{j+1},\ldots,k_{n}}, & \text{ if $\mathcal{M}_2$ is (\ref{M_2 in type B_n 1})},\\
\mathcal{M}^{(s)}_{0,\ldots,0,\underset{i}{k_i-1},0,\ldots,0,\underset{n}{k_n+2}}, & \text{ if $\mathcal{M}_2$ is (\ref{M_2 in type B_n 2})},\\
\mathcal{M}^{(s)}_{0,\ldots,0,\underset{i}{k_i-1},0,\ldots,0,\underset{n}{1}}, & \text{ if $\mathcal{M}_2$ is (\ref{M_2 in type B_n 3})}.
\end{cases}
\end{split}
\end{equation}

Let
\begin{align*}
&\mathcal{X}_i=\{(i,s+d_i(i+2l-5)): 1 \leq l \leq k_i \},
\end{align*}

and
\begin{align*}
\mathbb{X}_i=\prod_{(i,k)\in \mathcal{X}_i} \mathbb{X}_{i,k}^{i,k+2d_{i}}, \quad \mathbb{Y}_i=\prod_{(i,k)\in \mathcal{X}_i} \mathbb{Y}_{i,k}^{i,k+2d_{i}},
\end{align*}
where $\mathbb{X}_{i,k}^{i,k+2d_{i}}$, $\mathbb{Y}_{i,k}^{i,k+2d_{i}}$ are defined in Section \ref{Section neighbouring points}.

Let
\begin{align}\label{M_5 M_6 in types A_n B_n}
\begin{split}
&\mathcal{M}_5 =L((\prod_{(i,k)\in \mathbb{X}_i}Y_{i, k}) M_1 (M^{(s)}_{0, \ldots, 0, \underset{i}{k_i}, 0, \ldots, 0})^{-1}), \\
&\mathcal{M}_6 =L((\prod_{(i,k)\in \mathbb{Y}_i}Y_{i, k}) M_2 (M^{(s-2d_i)}_{0, \ldots, 0, \underset{i}{k_i}, 0, \ldots, 0})^{-1}),
\end{split}
\end{align}
where $M_1$, $M_2$ are the highest weight monomials of $\mathcal{M}_1$, $\mathcal{M}_2$ respectively.

We have the following theorem.
\begin{theorem}\label{M-systems}
In type $A_n$ (resp. $B_n$), we have the following system of equations
\begin{align}
[\mathcal{M}_1]  [\mathcal{M}_2] = [\mathcal{M}_3] [\mathcal{M}_4] + [\mathcal{M}_5] [\mathcal{M}_6] \label{general form of an equation in S-system}
\end{align}
in the Grothendieck $\mathcal{R}$ of the category $\mathcal{C}$, where $\mathcal{M}_1$ is defined in (\ref{M_1 in type A_n}) (resp. (\ref{M_1 in type B_n})), $\mathcal{M}_2$ is defined in (\ref{M_2 in type A_n}) (resp. (\ref{M_2 in type B_n 1})--(\ref{M_2 in type B_n 3})), $\mathcal{M}_3$, $\mathcal{M}_4$  are defined in (\ref{M_3 M_4 in type A_n}) (resp. (\ref{M_3 M_4 in type B_n})), $\mathcal{M}_5$, $\mathcal{M}_6$  are defined in (\ref{M_5 M_6 in types A_n B_n}).
\end{theorem}

We call the systems in Theorem \ref{M-systems} the M-systems of types $A_n$ and $B_n$. Theorem \ref{M-systems} will be proved in Section \ref{proof M-systems}.

\begin{remark}
By M-systems of types $A_n$, $B_n$, the $q$-characters of minimal affinizations of types $A_n$, $B_n$ can be computed recursively using the $q$-characters of Kirillov-Reshetikhin modules.
\end{remark}

\begin{remark}
The equations in the M-systems are different from the equations in the extended T-systems in \cite{MY12b} and the T-systems in \cite{KNS94}, \cite{Her06}. The union of the M-systems of types $A_n$, $B_n$, the dual M-systems of types $A_n$, $B_n$ (defined in Theorem \ref{dual M-systems}), and the T-systems of types $A_n$, $B_n$ is a closed system which contains all minimal affinizations of types $A_n$, $B_n$ (including Kirillov-Reshetikhin modules of types $A_n$, $B_n$).
\end{remark}

\begin{example}\label{eample1}
The following are some equations in the M-system of type $A_3$.
\begin{align*}
&[1_{-1}][1_{-3}2_{0}] = [1_{-3}1_{-1}][2_{0}] + [2_{-2}2_{0}],\\
&[1_{-3}1_{-1}][1_{-5}1_{-3}2_{0}] = [1_{-3}2_{0}][1_{-5}1_{-3}1_{-1}] + [2_{-4}2_{-2}2_{0}],\\
&[2_{-2}][2_{-4}3_{-1}] = [2_{-4}2_{-2}][3_{-1}] + [1_{-3}][3_{-3}3_{-1}],\\
&[2_{-4}2_{-2}][2_{-6} 2_{-4}3_{-1}] = [2_{-4}3_{-1}][ 2_{-6}2_{-4}2_{-2}] + [1_{-5}1_{-3}][3_{-5}3_{-3}3_{-1}],\\
&[1_{-3}][1_{-5}3_{-1}] = [1_{-5}1_{-3}][ 3_{-1}] + [2_{-4}3_{-1}],\\
&[1_{-5}1_{-3}][1_{-7}1_{-5}3_{-1}] = [1_{-5}3_{-1}][ 1_{-7}1_{-5}1_{-3}] + [2_{-6}2_{-4}3_{-1}],\\
&[1_{-5}3_{-1}][1_{-7}2_{-4}3_{-1}] = [1_{-7}1_{-5}3_{-1}][ 2_{-4}3_{-1}] + [ 2_{-6}2_{-4}3_{-1}][3_{-1}],\\
&[1_{-7}1_{-5}3_{-1}][1_{-9}1_{-7}2_{-4}3_{-1}] = [1_{-7}2_{-4}3_{-1}][1_{-9}1_{-7}1_{-5}3_{-1} ] +  [2_{-8}2_{-6}2_{-4}3_{-1}][ 3_{-1}].\\
\end{align*}

The following are some equations in the M-system of type $B_2$.
\begin{align*}
&[1_{-3}][1_{-7}2_{-2}2_{0}] = [1_{-7}1_{-3}][2_{-2}2_{0}] + [2_{-6}2_{-4}2_{-2}2_{0}],\\
&[1_{-7}1_{-3}][1_{-11}1_{-7}2_{-2}2_{0}] = [1_{-7}2_{-2}2_{0}][1_{-11}1_{-7}1_{-3}] +  [2_{-10}2_{-8}2_{-6}2_{-4}2_{-2}2_{0}],\\
&[1_{-1}][1_{-5}2_{0}]=[2_{0}][1_{-5}1_{-1}] + [2_{-4}2_{-2}2_{0}],\\
&[1_{-5}1_{-1}][1_{-9}1_{-5}2_{0}] = [1_{-5}2_{0}][1_{-9}1_{-5}1_{-1}] + [2_{-8}2_{-6}2_{-4}2_{-2}2_{0}].\\
\end{align*}
\end{example}

Moreover, we have the following theorem.
\begin{theorem}\label{irreducible}
The modules in the summands on the right hand side of each equation in Theorem \ref{M-systems} are simple.
\end{theorem}
Theorem \ref{irreducible} will be prove in Section \ref{proof irreducible}.

\subsection{The m-systems of types $A_n$, $B_{n}$} \label{m-systems Uqg}
Let
\[
\mathfrak{m}_i = \res(\mathcal{M}_{i})
\]
be the restriction of $\mathcal{M}_{i}$ to $U_q \mathfrak{g}$, where $1\leq i \leq 6$. Let $\chi(M)$ be the character of a $U_q \mathfrak{g}$-module $M$. We obtain a system of equations which we called the m-system of type $A_n$ (resp. $B_n$) from Theorem \ref{M-systems}:
\begin{gather}
\begin{split}
\chi( \mathfrak{m}_{1} ) \chi( \mathfrak{m}_{2} )=\chi( \mathfrak{m}_{3} )\chi( \mathfrak{m}_{4} )+\chi( \mathfrak{m}_{5} ) \chi( \mathfrak{m}_{6}).
\end{split}
\end{gather}

\section{Relation between M-systems and cluster algebras} \label{relation between M-systems and cluster algebras}
In this section, we will show that the equations in the M-system of type $A_{n}$ (resp. $B_{n}$) correspond to mutations in some cluster algebra $\mathscr{A}$ (resp. $\mathscr{A}'$) introduced in \cite{HL13}. Moreover, every minimal affinization in the M-system of type $A_{n}$ (resp. $B_{n}$) corresponds to a cluster variable in the cluster algebra $\mathscr{A}$ (resp. $\mathscr{A}'$).

\subsection{Definition of cluster algebras $\mathscr{A}$ and $\mathscr{A}'$}\label{definition of cluster algebra A}

We recall the definition of the cluster algebras introduced in \cite{HL13} in the following. Let $\widetilde{V}=I \times \mathbb{Z}$ and let $\widetilde{\Gamma}$ be a quiver with the vertex set $\widetilde{V}$ whose arrows are given by $(i,r) \to (j,s)$ if and only if $b_{ij} \neq 0$ and $s = r + b_{ij}$, where $B = (b_{ij})_{i,j \in I} = DC$ is defined in Section \ref{definition of quantum affine algebras}.

It is shown in \cite{HL13} that $\widetilde{\Gamma}$ has two isomorphic components. Let $\Gamma$ be one of the components and let $V$ be the set of vertices of $\Gamma$. Let $\psi$ be a function given by $\psi(i, r) = (i, d_i)$, $(i, r) \in V$. Let $W \subset I \times \mathbb{Z}$ be the image of $V$ under the map $\psi$ and let $G$ be the same quiver as $\Gamma$ but with vertices labelled by $W$. Let $W^-=W \cap (I \times \mathbb{Z}_{\leq 0})$ and let $Q$ be the full subquiver of $G$ with vertex set $W^-$.

Let $\mathbf{z}^-=\{z_{i,r}: (i,r)\in W^{-}\}$ and let $\mathscr{A}$ be the cluster algebra defined by the initial seed $(\mathbf{z}^-, Q)$. For convenience, we denote the quiver $Q$ and the cluster algebra $\mathscr{A}$ in the case of type $B_n$ by $Q'$ and $\mathscr{A}'$ respectively.

In order to have the correspondence between M-systems and cluster algebras, we use the following notations for the cluster variables in the initial seed. In the case of type $A_{n}$, let
\begin{align}\label{at}
{\bf m} =
     \{ m^{(-2k_{i}-i+3)}_{0,\ldots, 0,\underset{i}{k_{i}},0, \ldots, 0}\mid i \text{ is } \text{even}, \ k_{i} \in \mathbb{Z}_{\geq0} \}\cup
     \{ m^{(-2k_{i}-i+2)}_{0,\ldots, 0,\underset{i}{k_{i}},0, \ldots, 0}\mid i \text { is } \text{odd}, \ k_{i} \in \mathbb{Z}_{\geq0} \}.
\end{align}
In the case of type $B_{n}$, let ${\bf m}'={\bf m}_{1}\cup {\bf m}_{2}$, where
\begin{align}\label{bt1}
{\bf m}_{1} =\{ m^{(-2n-2k_{n}+5)}_{0,\ldots 0,\underset{n}{k_{n}}}\mid k_{n}\in \mathbb{Z}_{\geq0} \},
\end{align}
\begin{align}\label{bt2}
{\bf m}_{2} =\{ m^{(-4k_{i}-2i+3)}_{0,\ldots, 0,\underset{i}{k_{i}},0, \ldots, 0}, m^{(-4k_{i}-2i+5)}_{0,\ldots, 0,\underset{i}{k_{i}},0, \ldots, 0} \mid i\in \{1,\ldots,n-1\}, \ k_{i} \in \mathbb{Z}_{\geq0} \}.
\end{align}

The cluster algebra in the case of type $A_{n}$ (resp. $B_n$) is the cluster algebra defined by the initial seed $({\bf m}, Q)$ (resp. $({\bf m}', Q')$). Here we identify ${\bf m}$ (resp. ${\bf m}'$) with ${\bf z}^-$ as follows. Let $(i, r) \in W^-$ and $\delta_{i,n}$ the Kronecker delta. In the case of type $A_n$, we identify $m^{(r-i+1)}_{0,\ldots, 0,\underset{i}{k_{i}},0, \ldots, 0}$ with $z_{i, r}$. In the case of type $B_n$, we identify $m^{(r-2i+2+\delta_{i,n})}_{0,\ldots, 0,\underset{i}{k_{i}},0, \ldots, 0}$ with $z_{i, r}$.

We have the following theorem.
\begin{theorem} \label{minimal affinizations correspond to cluster variables}
Every minimal affinization in the M-system of type $A_n$ (resp. $B_n$) corresponds to a cluster variable in $\mathscr{A}$ (resp. $\mathscr{A}'$) defined in Section \ref{definition of cluster algebra A}.
\end{theorem}

We will prove Theorem \ref{minimal affinizations correspond to cluster variables} in Section \ref{proof of Theorem minimal affinizations correspond to cluster variables}.

\subsection{Mutation sequences: type $A_n$ case}\label{mutation sequence in type A_n}
We use the idea of the mutation sequences in \cite{HL13}. In \cite{HL13}, the mutation sequences produce Kirillov-Reshetikhin modules. In the following, the mutation sequences produce minimal affinizations.

When we say that we mutate $``C_{i}"$ of a quiver we mean that we mutate the first vertex in the $i$-th column of the quiver, then we mutate the second vertex in this column, and so on until the vertex at infinity. When we say that we mutate $(C_{i_1}, C_{i_2}, \ldots, C_{i_m} )$, $i_1, \ldots, i_m \in I$, we mean that we first mutate $C_{i_1}$, then we mutate $C_{i_2}$, and so on.
If $k_t = 0$, then ``we mutate $(C_{i_1},C_{i_2},\ldots,C_{i_m})$ $k_{t}$ times" means ``we do not mutate $(C_{i_1},C_{i_2},\ldots,C_{i_m})$".

We use $\emptyset$ to denote the empty mutation sequence and use
\begin{align*}
\prod_{k=1}^{m}(C_{2k}, C_{2k-1}, \ldots, C_{1})
\end{align*}
to denote the mutation sequence
\begin{align*}
(C_{2},C_{1}; C_4, C_3, C_2, C_1; \ldots; C_{2m}, C_{2m-1}, \ldots, C_{1}).
\end{align*}

Let $k_1, k_2, \ldots, k_n \in \mathbb{Z}_{\geq 0}$ and $k_i$ (resp. $k_r$) the first non-zero integer in $k_1, k_2, \ldots, k_n$ from the left (resp. right). Let
\begin{align*}
M_{r}^{(1)} =
\begin{cases}
    \emptyset,  & r=1,2,  \\
    \\
    \prod_{k=1}^{\frac{r-1}{2}}(C_{2k}, C_{2k-1}, C_{2k-2}, \ldots, C_{1}),  & r \equiv 1 \pmod2, r>1,  \\
    \\
     \prod_{k=1}^{\frac{r-2}{2}}(C_{2k},C_{2k-1},C_{2k-2},\ldots,C_{1}) , & r\equiv 0\pmod2, r>2.
\end{cases}
\end{align*}

Let $\seq$ be the mutation sequence: first we mutate $M_r^{(1)}$ starting from the initial quiver $Q$, then we mutate $(C_{r-1},C_{r-2},\ldots,C_{1})$ $k_{r}$ times, and then we mutate $(C_{r-2},C_{r-3},\ldots,C_{1})$ $k_{r-1}$ times; continue this procedure, we mutate $(C_{t-1},C_{t-2},\ldots,C_{1})$ $k_{t}$ times, $t=r-2, r-3, \ldots, i+1$. After we mutate $Q$ following the mutation sequence $\seq$, we obtain the minimal affinization (\ref{M_2 in type A_n}) at the $k_i$-th vertex of $i$-th column.

\subsection{Mutation sequences: type $B_n$ case}\label{mutation sequence in type B_n}
In the following, we define mutation sequences which produce the minimal affinizations (\ref{M_2 in type B_n 1}), (\ref{M_2 in type B_n 2}), (\ref{M_2 in type B_n 3}).

Let $k_1, k_2, \ldots, k_n \in \mathbb{Z}_{\geq 0}$ and $k_i$ (resp. $k_r$) the first non-zero integer in $k_1, k_2, \ldots, k_n$ from the left (resp. right). Let
\begin{align*}
N_{n}^{(1)} =
\begin{cases}
    \emptyset,   & n=2,  \\
    \\
    \prod_{k=0}^{\frac{n-3}{2}}(C_{2n-2k-1}, C_{2n-2k}, \ldots, C_{2n-1}),   & n \equiv 1 \pmod2, n\geq3, \\
    \\
    \prod_{k=0}^{\frac{n-4}{2}}(C_{2n-2k-2}, C_{2n-2k-1},\ldots, C_{2n-1}), & n\equiv 0  \pmod 2, n > 2,
     \end{cases}
\end{align*}

\begin{align*}
N_{n}^{(2)} =
\begin{cases}
    \emptyset,   &  n=2,3, \\
    \\
    \prod_{k=0}^{\frac{n-5}{2}}(C_{2k+2}, C_{2k+1}, \ldots, C_{1}),   & n \equiv 1 \pmod2, n>3, \\
    \\
    \prod_{k=0}^{\frac{n-4}{2}}(C_{2k+1}, C_{2k},\ldots, C_{1}), & n\equiv 0  \pmod 2, n>2,
     \end{cases}
\end{align*}

\begin{align*}
N_{n,r}^{(3)} = \begin{cases}
    \emptyset,    &    r=1,  \\
    \emptyset, & n \equiv 0 \pmod 2, r=2, \\
    \prod_{k=0}^{\frac{r-3}{2}}(C_{2n-2k-1}, C_{2n-2k}, \ldots, C_{2n-1}), & n \equiv 1 \pmod 2,  r\equiv1 \pmod 2,  r>1,\\
    \prod_{k=0}^{\frac{r-2}{2}}(C_{2n-2k-1}, C_{2n-2k}, \ldots, C_{2n-1}), & n \equiv 1 \pmod 2, r\equiv0 \pmod 2,  r\geq2,\\
   \prod_{k=0}^{\frac{r-4}{2}}(C_{2n-2k-2}, C_{2n-2k-1},\ldots, C_{2n-1}), & n \equiv 0 \pmod 2, r\equiv0 \pmod 2,       r>2, \\
    \prod_{k=0}^{\frac{r-3}{2}}(C_{2n-2k-2}, C_{2n-2k-1},\ldots, C_{2n-1}), & n \equiv 0 \pmod 2, r\equiv1 \pmod 2, r>1.
\end{cases}
\end{align*}

For $k_{n} \neq 0$ is even and $n$ is odd (resp. even), let $\seq$ be the mutation sequence: first we mutate $N_{n}^{(1)}$ starting from the initial quiver $Q'$, then we mutate $(C_{n+1},C_{n+2},\ldots,C_{2n-1})$ $\frac{k_{n}}{2}$ times,  and then we mutation $(C_{n+2}, C_{n+3}, \ldots, C_{2r-1})$ $k_{n-1}$ times; continue this procedure, we mutate $(C_{2n-t+1},C_{2n-t+2},\ldots,C_{2n-1})$ $k_{t}$ times, $t=n-2, n-3, \ldots,i+1$.

For $k_n \neq 0$ is odd, let $\seq$ be the mutation sequence: first we mutate $N_{n}^{(2)}$ starting from the initial quiver $Q'$, then we mutate $(C_{n-1},C_{n-2},\ldots,C_{1})$ $\frac{k_{n}+1}{2}$ times, and then we mutate $(C_{n-2},C_{n-3},\ldots,C_{1})$ $k_{n-1}$ times; continue this procedure, we mutate $(C_{t-1},C_{t-2},\ldots,C_{1})$ $k_{t}$ times, $t=n-2, n-3, \ldots,i+1$.

For $k_n=0$, let $k_r$ be the first non-zero integer from right in $k_1, k_2, \ldots, k_n$ and let $\seq$ be the mutation sequence: first we mutate $N_{n,r}^{(3)}$ starting from the initial quiver $Q'$, then we mutate $(C_{2n-r+1},C_{2n-r+2},\ldots,C_{2n-1})$ $k_{r}$ times, and then we mutation $(C_{2n-r+2}, C_{2n-r+3}, \ldots, C_{2n-1})$ $k_{r-1}$ times; continue this procedure, we mutate $(C_{2n-r-t+1},C_{2n-r-t+2},\ldots,C_{2n-1})$ $k_{t}$ times, $t=r-2, r-3, \ldots,i+1$.

After we mutate $Q'$ following the mutation sequence $\seq$, we obtain one of the minimal affinizations (\ref{M_2 in type B_n 1}), (\ref{M_2 in type B_n 2}), (\ref{M_2 in type B_n 3}) (the minimal affinization we obtain depends on $k_1, k_2, \ldots, k_n$) at the $k_i$-th vertex of $i$-th column.

\subsection{Proof of Theorem \ref{minimal affinizations correspond to cluster variables}} \label{proof of Theorem minimal affinizations correspond to cluster variables}

It suffices to prove that the equations in the M-system of type $A_n$ (resp. $B_{n}$) correspond to the mutations defined in Section \ref{mutation sequence in type A_n} (resp. Section \ref{mutation sequence in type B_n}).

Let $\mathfrak{M}$ be the set of minimal affinizations of type $A_n$ (resp. $B_n$). Let
\begin{align*}
\mathfrak{I}=\{ m^{(s)}_{k_1, k_2, \ldots,k_{n}}: s\in \mathbb{Z}, k_{1}, k_{2}, \ldots, k_{n} \in \mathbb{Z}_{\geq 0} \}.
\end{align*}

We define a map
\begin{equation}\label{S-s}
\begin{split}
\psi: \mathfrak{M}& \longrightarrow \mathfrak{I} \\
\mathcal{M}^{(s)}_{k_1, k_2, \ldots,k_{n}} & \mapsto m^{(s)}_{k_1, k_2, \ldots,k_{n}}.
\end{split}
\end{equation}
We apply the map $\psi$ defined by (\ref{S-s}) to the equations $[\mathcal{M}_1][\mathcal{M}_2]=[\mathcal{M}_3][\mathcal{M}_4]+[\mathcal{M}_5][\mathcal{M}_6]$ in the M-system for type $A_n$ (resp. $B_n$). Then we obtain a new system of equations:
\begin{align}
m_1m_2=m_3m_4+m_5m_6, \label{mutation equations 1}
\end{align}
where $m_i= \psi(\mathcal{M}_i)$, $1 \leq i \leq 6$.
For each equation in (\ref{mutation equations 1}), we define $m'_1=m_2$. Then we obtain a set of equations:
\begin{equation}\label{mutated equations}
\begin{split}
m'_1=\frac{m_3 m_4 + m_5 m_6}{m_1}.
\end{split}
\end{equation}
The above set of equations for type $A_n$ (resp. $B_n$) is the set of the mutation equations corresponding to the mutations in $\seq$ defined in Sections \ref{mutation sequence in type A_n} (resp. \ref{mutation sequence in type B_n}).

\section{The dual M-systems of types $A_n$ and $B_{n}$} \label{dual M system section}
In this section, we study the dual M-systems of types $A_n$ and $B_{n}$.

\subsection{The dual M-systems of types $A_n$, $B_{n}$}

\begin{lemma}[{Lemma 4.10, \cite{Her07}}] \label{map iota}
In the case of type $A_{n}$, let $\iota: \mathbb{Z}\mathcal{P}\rightarrow \mathbb{Z}\mathcal{P}$ be a homomorphism of rings such that $Y_{i, aq^{s}}\mapsto Y_{n-i+1, aq^{n-s+1}}^{-1}$ for all $i \in I$, $a\in \mathbb{C}^{\times}, s\in \mathbb{Z}$. Then
\begin{align*}
\chi_{q}(\widetilde{\mathcal M}_{k_{1},\ldots,k_{n}}^{(s)})=\iota(\chi_{q}(\mathcal{M}_{k_{1},\ldots,k_{n}}^{(s)})).
\end{align*}
In the case of type $B_{n}$, let $\iota: \mathbb{Z}\mathcal{P}\rightarrow \mathbb{Z}\mathcal{P}$ be a homomorphism of rings such that $Y_{i, aq^{s}} \mapsto Y_{i, aq^{4n-s-2}}^{-1}$ for all $i \in I$, $a\in \mathbb{C}^{\times}, s\in \mathbb{Z}$. Then
\begin{align*}
\chi_{q}(\widetilde{\mathcal M}_{k_{1},\ldots,k_{n}}^{(s)})=\iota(\chi_{q}(\mathcal{M}_{k_{1},\ldots,k_{n}}^{(s)})).
\end{align*}
\end{lemma}

\begin{theorem}\label{dual M-systems}
We have a system of equations
\begin{align}
[\widetilde{\mathcal{M}}_1]  [\widetilde{\mathcal{M}}_2] = [\widetilde{\mathcal{M}}_3] [\widetilde{\mathcal{M}}_4] + [\widetilde{\mathcal{M}}_5] [\widetilde{\mathcal{M}}_6]
\end{align}
in the Grothendieck $\mathcal{R}$ of the category $\mathcal{C}$, where $[\mathcal{M}_1]  [\mathcal{M}_2] = [\mathcal{M}_3] [\mathcal{M}_4] + [\mathcal{M}_5] [\mathcal{M}_6]$ are equations in the M-system of type $A_n$ (resp. $B_n$). Moreover, the modules in the
summands on the right hand side of each equation in the system are simple.
\end{theorem}
We call the systems in Theorem \ref{dual M-systems} dual M-systems of types $A_n$, $B_n$.

\begin{proof}
The dual M-system of type $A_n$ (resp. $B_n$) is obtained by applying $\iota$ defined in Lemma \ref{map iota} to both sides of every equation of the M-system of type $A_n$ (resp. $B_n$). The simplicity of the modules in the summands on the right hand side of each equation in the system follows from Theorem \ref{irreducible}.
\end{proof}

\begin{example}\label{eample1}
The following are some equations in the dual M-system of type $A_3$.
\begin{align*}
&[1_{1}][2_{0}1_{3}] = [1_{1}1_{3}][2_{0}] + [2_{0}2_{2}],\\
&[1_{1}1_{3}][2_{0}1_{3}1_{5}] = [2_{0}1_{3}][1_{1}1_{3}1_{5}] + [2_{0}2_{2}2_{4}],\\
&[2_{2}][3_{1}2_{4}] = [2_{2}2_{4}][3_{1}] + [1_{3}][3_{1}3_{3}],\\
&[2_{2}2_{4}][3_{1}2_{4}2_{6}] = [3_{1}2_{4}][ 2_{2}2_{4}2_{6}] + [1_{3}1_{5}][3_{1}3_{3}3_{5}],\\
&[1_{3}][3_{1}1_{5}] = [1_{3}1_{5}][3_{1}] + [3_{1}2_{4}],\\
&[1_{3}1_{5}][3_{1}1_{5}1_{7}] = [3_{1}1_{5}][ 1_{3}1_{5}1_{7}] + [3_{1}2_{4}2_{6}],\\
&[3_{1}1_{5}][3_{1}2_{4}1_{7}] = [3_{1}1_{5}1_{7}][ 3_{1}2_{4}] + [3_{1}2_{4}2_{6}][3_{1}],\\
&[3_{1}1_{5}1_{7}][3_{1}2_{4}1_{7}1_{9}] = [3_{1}2_{4}1_{7}][3_{1}1_{5}1_{7}1_{9}] +  [3_{1}2_{4}2_{6}2_{8}][ 3_{1}].\\
\end{align*}

The following are some equations in the M-system of type $B_2$.
\begin{align*}
&[1_{3}][2_{0}2_{2}1_{7}] = [1_{3}1_{7}][2_{0}2_{2}] + [2_{0}2_{2}2_{4}2_{6}],\\
&[1_{3}1_{7}][2_{0}2_{2}1_{7}1_{11}] = [2_{0}2_{2}1_{7}][1_{3}1_{7}1_{11}] +  [2_{0}2_{2}2_{4}2_{6}2_{8}2_{10}],\\
&[1_{1}][2_{0}1_{5}]=[2_{0}][1_{1}1_{5}] + [2_{0}2_{2}2_{4}],\\
&[1_{1}1_{5}][2_{0}1_{5}1_{9}] = [2_{0}1_{5}][1_{1}1_{5}1_{9}] + [2_{0}2_{2}2_{4}2_{6}2_{8}].\\
\end{align*}
\end{example}

\subsection{The dual m-systems of types $A_n$, $B_{n}$} \label{dual m-systems Uqg}
Let
\[
\widetilde{\mathfrak{m}}_i = \res(\widetilde{\mathcal{M}}_{i})
\]
be the restriction of $\widetilde{\mathcal{M}}_{i}$ to $U_q \mathfrak{g}$, where $1\leq i \leq 6$. We obtain a system of equations which we called the dual m-system of type $A_n$ (resp. $B_n$) from Theorem \ref{dual M-systems}:
\begin{gather}
\begin{split}
\chi( \widetilde{\mathfrak{m}}_{1} ) \chi( \widetilde{\mathfrak{m}}_{2} )=\chi( \widetilde{\mathfrak{m}}_{3} )\chi( \widetilde{\mathfrak{m}}_{4} )+\chi( \widetilde{\mathfrak{m}}_{5} ) \chi( \widetilde{\mathfrak{m}}_{6}).
\end{split}
\end{gather}

\subsection{Relation between dual M-systems and cluster algebras}
The following cluster algebra $\widetilde{\mathscr{A}}$ (resp. $\widetilde{\mathscr{A}}'$) is dual to the cluster algebra $\mathscr{A}$ (resp. $\mathscr{A}'$) defined in \cite{HL13}, see Section \ref{definition of cluster algebra A}.

Let $W^+=W \cap (I \times \mathbb{Z}_{\geq 0})$ and let $\widetilde{Q}$ be the full subquiver of $G$ with vertex set $W^+$, see Section \ref{definition of cluster algebra A}. Let $\mathbf{z}^+=\{z_{i,r}: (i,r)\in W^{+}\}$ and let $\widetilde{\mathscr{A}}$ be the cluster algebra defined by the initial seed $(\mathbf{z}^+, \widetilde{Q})$. For convenience, we denote this cluster algebra in the case of type $B_n$ by $\widetilde{\mathscr{A}}'$.

By similar arguments in Section \ref{relation between M-systems and cluster algebras}, we have the following theorem.
\begin{theorem}\label{minimal affinizations correspond to cluster variablesII}
Every equation in the dual M-system of type $A_n$ (resp. $B_{n}$) corresponds to a mutation equation in the cluster algebra $\widetilde{\mathscr{A}}$ (resp. $\widetilde{\mathscr{A}}'$). Every minimal affinization in the dual M-system of type $A_n$ (resp. $B_{n}$) corresponds to a cluster variable of the cluster algebra $\widetilde{\mathscr{A}}$ (resp. $\widetilde{\mathscr{A}}'$).
\end{theorem}

\section{Connection with the Hernandez-Leclerc conjecture} \label{section Hernandez-Leclerc conjecture}
In this section, we show that our results imply that the Hernandez-Leclerc conjecture (Conjecture \ref{Hernandez-Leclerc conjecture 1}) is true for minimal affinizations of types $A_n$ and $B_n$.

Let us recall the definition of $\mathcal{C}_{\ell}$, see \cite{HL10}. Let $I$ be the set of vertices of the Dynkin diagram of $\mathfrak{g}$. The Drinfeld polynomials of a simple $U_q \widehat{\mathfrak{g}}$-module $S$ is an $I$-tuple of polynomials $\pi_S = (\pi_{i,S}(u); i\in I)$ in one indeterminate $u$ with coefficients in $\mathbb{C}$ and constant term $1$. Let $S$ be a simple $U_q \widehat{\mathfrak{g}}$-module with Drinfeld polynomials $\pi_{i,S}(u)=\prod_{k=1}^{n_i}(1-ua_{k}^{(i)})$, $i\in I$. Then the highest weight monomial of $S$ is $m_S=\prod_{i \in I} \prod_{k=1}^{n_i} Y_{i,a_k^{(i)}}$.

Let $I = I_0 \sqcup I_1$ be a partition of $I$ such that every edge connects a vertex of $I_0$ with a vertex of $I_1$. For $i \in I$, let $\xi_i = 0$ if $i \in I_0$ and $\xi_i=1$ if $i \in I_1$.

Let $\mathcal{C}_{\ell}$ ($\ell \in \mathbb{Z}_{\geq 0}$) be the full subcategory of $\mathcal{C}$ whose objects $V$ satisfy: for every composition factor $S$ of $V$ and every $i\in I$, the roots of the Drinfeld polynomial $\pi_{i,S}(u)$ belong to $\{q^{-2k-\xi_i}| 0 \leq k \leq \ell \}$.

\begin{remark}
The category $\mathcal{C}_{\ell}$ used in \cite{HL13} is slightly different. It can be defined as follows. Let $\mathcal{C}_{\ell}$ ($\ell \in \mathbb{Z}_{\leq 0}$) be the full subcategory of $\mathcal{C}$ whose objects $V$ satisfy: for every composition factor $S$ of $V$ and every $i\in I$, the roots of the Drinfeld polynomial $\pi_{i,S}(u)$ belong to $\{q^{2k+\xi_i}| \ell \leq k \leq 0 \}$. In this paper, we also use this definition of $\mathcal{C}_{\ell}$. The minimal affinizations corresponding to the cluster variables obtained from the mutations defined in Sections \ref{mutation sequence in type A_n} and \ref{mutation sequence in type B_n} are in $\mathcal{C}_{\ell}$.
\end{remark}

Conjecture \ref{Hernandez-Leclerc conjecture 1} has been proved in \cite{HL13} for Kirillov-Reshetikhin modules in all types. We have the following theorem.
\begin{theorem} \label{minimal affinizations are real prime simple and correspond to cluster variables}
Minimal affinizations of type $A_n$ (resp. $B_n$) are simple, real, prime and they correspond to cluster variables in $\mathscr{A}$, $\widetilde{\mathscr{A}}$ (resp. $\mathscr{A}'$, $\widetilde{\mathscr{A}'}$). Therefore Conjecture \ref{Hernandez-Leclerc conjecture 1} is true for minimal affinizations in types $A_n$, $B_n$.
\end{theorem}
By Theorem \ref{minimal affinizations correspond to cluster variables} and Theorem \ref{minimal affinizations correspond to cluster variablesII}, every minimal affinization in the M-system of type $A_n$ (resp. $B_n$) corresponds to a cluster variable in $\mathscr{A}$ (resp. $\mathscr{A}'$), every minimal affinization in the dual M-system of type $A_n$ (resp. $B_n$) corresponds to a cluster variable in $\widetilde{\mathscr{A}}$ (resp. $\widetilde{\mathscr{A}}'$). By the results in \cite{CMY13}, minimal affinizations of all Dynkin types are prime. By definition, minimal affinizations are simple. Therefore to prove Theorem \ref{minimal affinizations are real prime simple and correspond to cluster variables}, we only need to show that minimal affinizations of types $A_n$ and $B_n$ are real. We have the following theorem.

\begin{theorem} \label{minimal affinizations of types AB are real}
Minimal affinizations of types $A_n$ and $B_n$ are real.
\end{theorem}
\begin{proof}
The theorem follows from the following facts:
\begin{enumerate}[(1)]
\item $\chi_{q}(\mathcal{M}^{(s)}_{k_{1},k_{2},\ldots,k_{n}})\chi_{q}(\mathcal{M}^{(s)}_{k_{1},k_{2},\ldots,k_{n}})$, $s \in \mathbb{Z}, \ k_1, \ldots, k_n \in \mathbb{Z}_{\geq 0}$, has only one dominant monomial $M^{(s)}_{k_{1},k_{2},\ldots,k_{n}}M^{(s)}_{k_{1},k_{2},\ldots,k_{n}}$;

\item $\chi_{q}(\widetilde{\mathcal{M}}^{(s)}_{k_{1},k_{2},\ldots,k_{n}})\chi_{q}(\widetilde{\mathcal{M}}^{(s)}_{k_{1},k_{2},\ldots,k_{n}})$, $s \in \mathbb{Z}, \ k_1, \ldots, k_n \in \mathbb{Z}_{\geq 0}$, has only one dominant monomial $\widetilde{M}^{(s)}_{k_{1},k_{2},\ldots,k_{n}}\widetilde{M}^{(s)}_{k_{1},k_{2},\ldots,k_{n}}$.
\end{enumerate}
We will prove (1) in the case of type $A_{n}$. The other cases are similar.

Let $k_{i_1}, \ldots, k_{i_r}$ be the non-zero integers in $k_1, \ldots, k_n$. Then the length of $\mathcal{M}^{(s)}_{k_{1},k_{2},\ldots,k_{n}}$ is $L=k_{i_{1}}+k_{i_{2}}+\cdots+k_{i_{r}}$. Let $\mathbf{m}=\prod^{L}_{t=1}m(p_{t})$ (resp. $\mathbf{m}'=\prod^{L}_{t=1}m(p'_{t})$) be a monomial in the first (resp. the second) $\chi_{q}(\mathcal{M}^{(s)}_{k_{1},k_{2},\ldots,k_{n}})$ in $\chi_{q}(\mathcal{M}^{(s)}_{k_{1},k_{2},\ldots,k_{n}})\chi_{q}(\mathcal{M}^{(s)}_{k_{1},k_{2},\ldots,k_{n}})$, where $m$ is the map defined by (\ref{map sending paths to monomials}), $(p_{1},\ldots,p_{L})\in \overline{\varphi}_{(c_{t},d_{t})_{1\leq t\leq L}}$ (resp. $(p'_{1},\ldots,p'_{L})\in \overline{\varphi}_{(c_{t},d_{t})_{1\leq t\leq L}}$) is a tuple of non-overlapping paths, $d_1, \ldots, d_L$ are some integers determined by $s, k_1, k_2, \ldots, k_n$, and
\begin{align*}
&c_{1}=c_{2}=\cdots=c_{k_{i_{1}}}=i_{1}, \ c_{k_{i_{1}}+1}=c_{k_{i_{1}}+2}=\cdots=c_{k_{i_{1}}+k_{i_{2}}}=i_{2}, \ \ldots, \\
&\ldots, \ c_{k_{i_{1}}+k_{i_{2}}+\cdots+k_{i_{r-1}}+1}=c_{k_{i_{1}}+k_{i_{2}}+\cdots+k_{i_{r-1}}+2}=\cdots=c_{k_{i_{1}}+k_{i_{2}}+\cdots+k_{i_{r-1}}+k_{i_{r}}}=i_{r}.
\end{align*}

Suppose that $\mathbf{m} \mathbf{m}'$ is dominant. If $p_{L}\neq p^{+}_{c_{L},d_{L}}$ in $\mathbf{m}$, then $\mathbf{m} \mathbf{m}'$ is right-negative and hence not dominant. Therefore $p_{L}= p^{+}_{c_{L},d_{L}}$. By Theorem \ref{path description of q-characters}, we have $p_{u}=p^{+}_{c_{u},d_{u}}$, $\sum_{\ell=1}^{r-1}k_{i_{\ell}}< u<L$. Similarly, for $\sum_{\ell=1}^{r-1}k_{i_{\ell}}< u\leq L$, we have $p'_{u}=p'^{+}_{c_{u},d_{u}}$.

Suppose that $p_{\sum_{\ell=1}^{r-1}k_{i_{\ell}}}\neq p^{+}_{c_{\sum_{\ell=1}^{r-1}k_{i_{\ell}}},d_{\sum_{\ell=1}^{r-1}k_{i_{\ell}}}}$. Then $m(p_{\sum_{\ell=1}^{r-1}k_{i_{\ell}}})$ has some negative factor $h_b^{-1}$, where $(h, b) \in C_{p_{\sum_{\ell=1}^{r-1}k_{i_{\ell}}}}^-$. By Theorem \ref{path description of q-characters}, $\mathbf{m}$ has the factor $h_b^{-1}$. Therefore $h_b^{-1}$ is cancelled by $\mathbf{m}'$. It follows that ${\bf m}' \neq M^{(s)}_{k_{1},k_{2},\ldots,k_{n}}$ since $h_b$ is not in $M^{(s)}_{k_{1},k_{2},\ldots,k_{n}}$. But then $\mathbf{m} \mathbf{m}'$ has one of the factors
\begin{align}
1_{h+b-1}^{-1}, \ 2_{h+b-2}^{-1}, \ \ldots, \ (h-2)_{b+2}^{-1}, \ (h-1)_{b+1}^{-1}, \ (h+1)_{b+1}^{-1}, \ (h+2)_{b+2}^{-1}, \ \ldots, \ n_{b+n-h}^{-1}.
\end{align}
This contradicts the assumption that $\mathbf{m} \mathbf{m}'$ is dominant. Therefore $p_{\sum_{\ell=1}^{r-1}k_{i_{\ell}}}= p^{+}_{c_{\sum_{\ell=1}^{r-1}k_{i_{\ell}}},d_{\sum_{\ell=1}^{r-1}k_{i_{\ell}}}}$. By Theorem \ref{path description of q-characters}, we have $p_{t}=p^{+}_{c_{t},d_{t}}$, $\sum_{\ell=1}^{r-2}k_{i_{\ell}}<t\leq \sum_{\ell=1}^{r-1}k_{i_{\ell}}$. By the same arguments, we have $p_{t}=p^{+}_{c_{t},d_{t}}$ for $1 \leq t\leq\sum_{\ell=1}^{r-2}k_{i_{\ell}}$. Therefore ${\bf m} = M^{(s)}_{k_{1},k_{2},\ldots,k_{n}}$.

By the same arguments, we have ${\bf m}' = M^{(s)}_{k_{1},k_{2},\ldots,k_{n}}$. Therefore the only dominant monomial in $\chi_{q}(\mathcal{M}^{(s)}_{k_{1},k_{2},\ldots,k_{n}})\chi_{q}(\mathcal{M}^{(s)}_{k_{1},k_{2},\ldots,k_{n}})$ is $M^{(s)}_{k_{1},k_{2},\ldots,k_{n}}M^{(s)}_{k_{1},k_{2},\ldots,k_{n}}$.
\end{proof}

\section{Proof of Theorem \ref{M-systems}} \label{proof M-systems}
In this section, we will prove Theorem \ref{M-systems}.

\subsection{Classification of dominant monomials}
First we classify all dominant monomials in each summand on the left and right hand sides of every equation in Theorem \ref{M-systems}. We have the following lemma.
\begin{lemma} \label{set of dominant monomials in terms monomials}
Let $[\mathcal{M}_1]  [\mathcal{M}_2] = [\mathcal{M}_3] [\mathcal{M}_4] + [\mathcal{M}_5] [\mathcal{M}_6]$ be any equation in the M-system of type $A_n$ (resp. $B_n$) in Theorem \ref{M-systems}. Let $M_i$ be the highest $l$-weight monomial of $\mathcal{M}_i$, $i \in \{1, 2, \ldots, 6\}$. The dominant monomials in each summand on the left and right hand sides of $[\mathcal{M}_1]  [\mathcal{M}_2] = [\mathcal{M}_3] [\mathcal{M}_4] + [\mathcal{M}_5] [\mathcal{M}_6]$ are given in Table \ref{dominant monomials in the M-system of type A} (resp. Table \ref{dominant monomials in the M-system of type B}).
\end{lemma}
We will prove Lemma \ref{set of dominant monomials in terms monomials} in Section \ref{proof classification of dominant monomials}.

In Table \ref{dominant monomials in the M-system of type A} and Table \ref{dominant monomials in the M-system of type B}, $M\prod_{0\leq j\leq r}A^{-1}_{i,s} = M$ for $r = -1$, $s \in \mathbb{Z}$.

\begin{table}[H] \resizebox{.6\width}{.6\height}{
\begin{tabular}{|c|c|c|c|}
\hline %
$\mathcal{M}_2$ & Dominant monomials of $\chi_q(\mathcal{M}_1)\chi_q(\mathcal{M}_2)$ & Dominant monomials of $\chi_q(\mathcal{M}_3)\chi_q(\mathcal{M}_4)$ & Dominant monomials of $\chi_q(\mathcal{M}_5)\chi_q(\mathcal{M}_6)$ \\
\hline %
$\mathcal{M}^{(s-2)}_{k_{1},k_{2}+1,k_{3},\ldots,k_{n}}$ & $\substack{M_1M_2\prod_{0\leq j\leq r}A^{-1}_{1,s+2k_{1}-2j-3},\\ -1\leq r\leq k_{1}-1}$ & $\substack{M_3M_4\prod_{0\leq j\leq r}A^{-1}_{1,s+2k_{1}-2j-3},\\ -1\leq r\leq k_{1}-2}$ & $\substack{M_5M_6}$ \\
\hline %
$\mathcal{M}^{(s-2)}_{0,\ldots, 0,\underset{i}{k_{i}},k_{i+1}+1,k_{i+2},\ldots,k_{n}}$ & $\substack{M_1M_2\prod_{0\leq j\leq r}A^{-1}_{i,s+i+2k_{i}-2j-4},\\ -1\leq r\leq k_{i}-1}$ & $\substack{M_3M_4\prod_{0\leq j\leq r}A^{-1}_{i,s+i+2k_{i}-2j-4}, \\-1\leq r\leq k_{i}-2}$ & $ \substack{M_5M_6}$ \\
\hline %
$\mathcal{M}^{(s-2)}_{k_{1},0,\ldots,0,k_{j}+1,k_{j+1},\ldots,k_{n}}$ & $\substack{M_1M_2\prod_{0\leq j\leq r}A^{-1}_{1,s+2k_{1}-2j-3}, \\-1\leq r\leq k_{1}-1 }$  & $\substack{M_3M_4\prod_{0\leq j\leq r}A^{-1}_{1,s+2k_{1}-2j-3},\\ -1\leq r\leq k_{1}-2}$ & $\substack{M_5M_6}$ \\
\hline %
$\mathcal{M}^{(s-2)}_{0,\ldots, 0,\underset{i}{k_{i}},0,\ldots,0,\underset{j}{k_{j}+1},k_{j+1},\ldots,k_{n}}$ & $\substack{M_1M_2\prod_{0\leq j\leq r}A^{-1}_{i,s+i+2k_{i}-2j-4},\\ -1\leq r\leq k_{i}-1}$  & $ \substack{M_3M_4\prod_{0\leq j\leq r}A^{-1}_{i,s+i+2k_{i}-2j-4}, \\-1\leq r\leq k_{i}-2} $ & $ \substack{M_5M_6}$ \\
\hline %
\end{tabular} }\vskip 0.05in
\caption{Classification of dominant monomials in the M-system of type $A_{n}$.}
\label{dominant monomials in the M-system of type A}
\end{table}

\begin{table}[H] \resizebox{.6\width}{.6\height}{
\begin{tabular}{|c|c|c|c|}
\hline %
$\mathcal{M}_2$ & Dominant monomials of $\chi_q(\mathcal{M}_1)\chi_q(\mathcal{M}_2)$ & Dominant monomials of $\chi_q(\mathcal{M}_3)\chi_q(\mathcal{M}_4)$ & Dominant monomials of $\chi_q(\mathcal{M}_5)\chi_q(\mathcal{M}_6)$ \\
\hline %
$\mathcal{M}_{k_{1},k_{2}+1,k_{3},\ldots,k_{n}}^{(s-4)}$ & $\substack{M_1M_2\prod_{0 \leq j \leq r}A^{-1}_{1,s+4k_{1}-4j-6}, \\-1 \leq r \leq k_{1}-1}$ & $\substack{M_3M_4\prod_{0 \leq j \leq r}A^{-1}_{1,s+4k_{1}-4j-6},\\ -1 \leq r \leq k_{1}-2}$ & $ \substack{M_5M_6}$ \\
\hline %
$\mathcal{M}_{0,\ldots,0,\underset{i}{k_{i}},k_{i+1}+1,k_{i+2},\ldots,k_{n}}^{(s-4)}$ & $\substack{M_1M_2\prod_{0 \leq j \leq r}A^{-1}_{i,s+2i+4k_{i}-4j-8}, \\-1 \leq r \leq k_{i}-1}$ & $\substack{M_3M_4\prod_{0 \leq j \leq r}A^{-1}_{i,s+2i+4k_{i}-4j-8}, \\-1 \leq r \leq k_{i}-2}$ & $\substack{M_5M_6}$ \\
\hline %
$\mathcal{M}_{0,\ldots,0,k_{n-1},k_{n}+2}^{(s-4)}$ & $\substack{M_1M_2\prod_{0 \leq j \leq r}A^{-1}_{n-1,s+2n+4k_{n-1}-4j-10}, \\-1 \leq r \leq k_{n-1}-1}$  & $\substack{M_3M_4\prod_{0 \leq j \leq r}A^{-1}_{n-1,s+2n+4k_{n-1}-4j-10},\\ -1 \leq r \leq k_{n-1}-2}$ & $\substack{M_5M_6} $ \\
\hline %
$\mathcal{M}_{0,\ldots,0,k_{n-1},1}^{(s-4)}$ & $\substack{M_1M_2\prod_{0 \leq j \leq r}A^{-1}_{n-1,s+2n+4k_{n-1}-4j-10}, \\-1 \leq r \leq k_{n-1}-1}$  & $\substack{M_3M_4\prod_{0 \leq j \leq r}A^{-1}_{n-1,s+2n+4k_{n-1}-4j-10},\\ -1 \leq r \leq k_{n-1}-2}$ & $\substack{M_5M_6} $ \\
\hline %
$\mathcal{M}_{k_{1},0,\ldots,0,k_{j}+1,k_{j+1},\ldots,k_{n}}^{(s-4)}$ & $\substack{M_1M_2\prod_{0 \leq j \leq r}A^{-1}_{1,s+4k_{1}-4j-6},\\ -1 \leq r \leq k_{1}-1}$  & $ \substack{M_3M_4\prod_{0 \leq j \leq r}A^{-1}_{1,s+4k_{1}-4j-6},\\ -1 \leq r \leq k_{1}-2} $ & $ \substack{M_5M_6}$\\
\hline %
$\mathcal{M}_{k_{1},0,\ldots,0,k_{n}+2}^{(s-4)}$ & $\substack{M_1M_2\prod_{0 \leq j \leq r}A^{-1}_{1,s+4k_{1}-4j-6}, \\-1 \leq r \leq k_{1}-1}$  & $ \substack{M_3M_4 \prod_{0 \leq j \leq r}A^{-1}_{1,s+4k_{1}-4j-6},\\ -1 \leq r \leq k_{1}-2} $ & $\substack{M_5M_6}$\\
\hline %
$\mathcal{M}_{k_{1},0,\ldots,0,1}^{(s-4)}$ & $\substack{M_1M_2\prod_{0 \leq j \leq r}A^{-1}_{1,s+4k_{1}-4j-6}, \\-1 \leq r \leq k_{1}-1}$  & $ \substack{M_3M_4 \prod_{0 \leq j \leq r}A^{-1}_{1,s+4k_{1}-4j-6},\\ -1 \leq r \leq k_{1}-2} $ & $\substack{M_5M_6}$\\
\hline %
$\mathcal{M}_{0,\ldots,0,\underset{i}{k_{i}},0,\ldots,0,k_{j}+1,k_{j+1},\ldots,k_{n}}^{(s-4)}$ & $\substack{M_1M_2\prod_{0 \leq j \leq r}A^{-1}_{i,s+2i+4k_{i}-4j-8},\\ -1 \leq r \leq k_{i}-1}$  & $ \substack{M_3M_4\prod_{0 \leq j \leq r}A^{-1}_{i,s+2i+4k_{i}-4j-8}, \\-1 \leq r \leq k_{i}-2} $ & $\substack{M_5M_6}$\\
\hline %
$\mathcal{M}_{0,\ldots,0,\underset{i}{k_{i}},0,\ldots,0,k_{n}+2}^{(s-4)}$ & $\substack{M_1M_2\prod_{0 \leq j \leq r}A^{-1}_{i,s+2i+4k_{i}-4j-8},\\ -1 \leq r \leq k_{i}-1}$  & $ \substack{M_3M_4\prod_{0 \leq j \leq r}A^{-1}_{i,s+2i+4k_{i}-4j-8},\\ -1 \leq r \leq k_{i}-2}$ & $\substack{M_5M_6}$\\
\hline %
$\mathcal{M}_{0,\ldots,0,\underset{i}{k_{i}},0,\ldots,0,1}^{(s-4)}$ & $\substack{M_1M_2\prod_{0 \leq j \leq r}A^{-1}_{i,s+2i+4k_{i}-4j-8},\\ -1 \leq r \leq k_{i}-1}$  & $ \substack{M_3M_4\prod_{0 \leq j \leq r}A^{-1}_{i,s+2i+4k_{i}-4j-8},\\ -1 \leq r \leq k_{i}-2}$ & $\substack{M_5M_6}$\\
\hline %
\end{tabular}}
\caption{Classification of dominant monomials in the M-system of type $B_{n}$.}
\label{dominant monomials in the M-system of type B}
\end{table}

\subsection{Proof of Theorem \ref{M-systems}}
By Table \ref{dominant monomials in the M-system of type A} and Table \ref{dominant monomials in the M-system of type B}, the dominant monomials in the $q$-characters of the left hand side and of the right hand side of every equation in Theorem \ref{M-systems} are the same. Therefore  Theorem \ref{M-systems} is true.

\subsection{Proof of Lemma \ref{set of dominant monomials in terms monomials}} \label{proof classification of dominant monomials}
We will prove the case of the $4$-th line of the second column of Table \ref{dominant monomials in the M-system of type A} and the case of $3$-rd line of the second column of Table \ref{dominant monomials in the M-system of type B}. The other cases are similar.

\begin{proof}[\bf Proof of the case of the $4$-th line of the second column of Table \ref{dominant monomials in the M-system of type A}.]
We have
\begin{align*}
M_1 & = M^{(s)}_{0,\ldots,0,\underset{i}{k_{i}},0,\ldots,0,\underset{j}{k_{j}},k_{j+1},\ldots,k_{n}}), \\
M_2 & = M^{(s-2)}_{0,\ldots,0,\underset{i}{k_{i}},0,\ldots,0,\underset{j}{k_{j}+1},k_{j+1},\ldots,k_{n}}.
\end{align*}

Let $L=k_{i}+k_{j}+\cdots+k_{n}$. Let ${\bf m}=\prod^{L}_{t=1}m(p_{t})$ be a monomial in $\chi_{q}(M_1)$, where $(p_{1},\ldots,p_{L})\in \overline{\varphi}_{(c_{t},d_{t})_{1\leq t\leq L}}$ is a tuple of non-overlapping paths, $d_1, \ldots, d_L$ are some integers determined by $s, k_1, k_2, \ldots, k_n$, and
\begin{align*}
&c_{1}=c_{2}=\cdots=c_{k_{i}}=i, \ c_{k_{i}+1}=c_{k_{i}+2}=\cdots=c_{k_{i}+k_{j}}=j, \ \ldots, \\
&\ldots, \ c_{k_{i}+k_{j}+\cdots+k_{n-1}+1}=c_{k_{i}+k_{j}+\cdots+k_{n-1}+2}=\cdots=c_{k_{i}+k_{j}+\cdots+k_{n}}=n.
\end{align*}
Let ${\bf m}'=\prod^{L+1}_{u=1}m(p'_{u})$ be a monomial in $\chi_{q}(M_2)$, where $(p'_{1},\ldots,p'_{L+1})\in \overline{\varphi}_{(c'_{u},d'_{u})_{1\leq u\leq L+1}}$ is a tuple of non-overlapping paths, $d'_{1}, \ldots, d_{L+1}'$ are some integers determined by $s, k_1, k_2, \ldots, k_n$, and
\begin{align*}
&c'_{1}=c'_{2}=\cdots=c'_{k_{i}}=i, \ c'_{k_{i}+1}=c'_{k_{i}+2}=\cdots=c'_{k_{i}+k_{j}}=c'_{k_{i}+k_{j}+1}=j, \ \ldots, \\
&\ldots, \ c'_{k_{i}+k_{j}+\cdots+k_{n-1}+2}=c'_{k_{i}+k_{j}+\cdots+k_{n-1}+3}=\cdots=c'_{k_{i}+k_{j}+\cdots+k_{n}+1}=n.
\end{align*}

Suppose that ${\bf m} {\bf m}'$ is dominant. The length of $M_1$ is $L$ and the length of $M_2$ is $L+1$. By the same arguments as the arguments in the proof of Theorem \ref{minimal affinizations of types AB are real}, we have $p'_{u}=p'^{+}_{c'_{u},d'_{u}}$, $k_{i}< u \leq L+1$ and $p_{t}=p^{+}_{c_{t},d_{t}}$, $1 \leq t\leq L$. Therefore ${\bf m} = M_1$.

If $p'_{k_i} = p'^+_{c_{k_{i}},d_{k_{i}}}$, then $p'_{k_i - \ell} = p'^+_{c_{k_i - \ell},d_{k_i - \ell}}$ $(1\leq \ell \leq k_i-1)$. Therefore ${\bf m} {\bf m}' = M_1 M_2$. If $p'_{k_i} = p'^+_{c_{k_{i}},d_{k_{i}}} A^{-1}_{i,s+i+2k_{i}-4}$, then $p'_{k_i - \ell} \in \{ p'^+_{c_{k_i - \ell},d_{k_i - \ell}}, p'^+_{c_{k_i - \ell},d_{k_i - \ell}} A^{-1}_{i, s+i+2k_{i}-2\ell-4} \}$ $(1\leq \ell \leq k_i-1)$. Therefore ${\bf m} {\bf m}'$ is one of the dominant monomials $M_1 M_2 \prod_{0\leq j\leq r}A^{-1}_{i, s+i+2k_{i}-2j-4}$, $0\leq r\leq k_{i}-1$. If $p'_{k_i} \not\in \{ p'^+_{c_{k_{i}},d_{k_{i}}}, p'^+_{c_{k_{i}},d_{k_{i}}} A^{-1}_{i, s+i+2k_{i}-4}\}$, then by the same arguments as the arguments in the proof of Theorem \ref{minimal affinizations of types AB are real}, ${\bf m} {\bf m}'$ is not dominant which contradicts our assumption.
\end{proof}

\begin{proof}[\bf Proof of the case of the $3$-rd line of the second column of Table \ref{dominant monomials in the M-system of type B}.]
We have
\begin{align*}
M_1 = M_{0,\ldots,0,k_{n-1},k_{n}}^{(s)}), \quad M_2 = M_{0,\ldots,0,k_{n-1},k_{n}+2}^{(s-4)}).
\end{align*}

Let $L=k_{n-1}+k_{n}$. Let ${\bf m}=\prod^{L}_{t=1}m(p_{t})$ be a monomial in $\chi_{q}(M_1)$, where $(p_{1},\ldots,p_{L})\in \overline{\varphi}_{(c_{t},d_{t})_{1\leq t\leq L}}$ is a tuple of non-overlapping paths, $d_1, \ldots, d_L$ are some integers determined by $s, k_{n-1}, k_{n}$, and
\begin{align*}
&c_{1}=c_{2}=\cdots=c_{k_{n-1}}=n-1, \ c_{k_{n-1}+1}=c_{k_{n-1}+2}=\cdots=c_{k_{n-1}+k_{n}}=n.
\end{align*}
Let ${\bf m}'=\prod^{L+2}_{u=1}m(p'_{u})$ be a monomial in $\chi_{q}(M_2)$, where
\begin{align*}
(p'_{1},\ldots,p'_{L+2})\in \overline{\varphi}_{(c'_{u},d'_{u})_{1\leq u\leq L+2}}
\end{align*}
is a tuple of non-overlapping paths, $d'_{1}, \ldots, d_{L+2}'$ are some integers determined by $s, k_{n-1}, k_{n}$, and
\begin{align*}
&c'_{1}=c'_{2}=\cdots=c'_{k_{n-1}}=n-1, \ c'_{k_{n-1}+1}=c'_{k_{n-1}+2}=\cdots=c'_{k_{n-1}+k_{n}+2}=n.
\end{align*}
Suppose that ${\bf m} {\bf m}'$ is dominant. The length of $M_1$ is $L$ and the length of $M_2$ is $L+2$. By the same arguments as the arguments in the proof of Theorem \ref{minimal affinizations of types AB are real}, we have $p'_{u}=p'^{+}_{c'_{u},d'_{u}}$, $k_{n-1} < u \leq L+2$ and $p_{t}=p^{+}_{c_{t},d_{t}}$, $1 \leq t\leq L$. Therefore ${\bf m} = M_1$.

If $p'_{k_n-1} = p'^+_{c_{k_n-1},d_{k_n-1}}$, then $p'_{k_{n-1} - \ell} = p'^+_{c_{k_{n-1} - \ell},d_{k_{n-1} - \ell}}$ $(1\leq \ell \leq k_{n-1}-1)$. Therefore ${\bf m} {\bf m}' = M_1 M_2$. If $p'_{k_{n-1}} = p'^+_{c_{k_{n-1}},d_{k_{n-1}}} A^{-1}_{n-1, s+2n+4k_{n-1}-10}$, then
\begin{gather}
\begin{align*}
p'_{k_{n-1} - \ell} \in \{ p'^+_{c_{k_{n-1} - \ell},d_{k_{n-1} - \ell}}, p'^+_{c_{k_{n-1} - \ell},d_{k_{n-1} - \ell}} A^{-1}_{n-1, s+2n+4k_{n-1}-4\ell-10} \},  \quad 1\leq \ell \leq k_{n-1}-1.
\end{align*}
\end{gather}
Therefore ${\bf m} {\bf m}'$ is one of the dominant monomials $M_1 M_2 \prod_{0 \leq j \leq r}A^{-1}_{n-1, s+2n+4k_{n-1}-4j-10}$, $0 \leq r \leq k_{n-1}-1$. If $p'_{k_{n-1}} \not\in \{ p'^+_{c_{k_{n-1}},d_{k_{n-1}}},  p'^+_{c_{k_{n-1}},d_{k_{n-1}}} A^{-1}_{n-1, s+2n+4k_{n-1}-10}\}$, then by the same arguments as the arguments in the proof of Theorem \ref{minimal affinizations of types AB are real}, ${\bf m} {\bf m}'$ is not dominant which contradicts our assumption.
\end{proof}

\section{Proof of Theorem \ref{irreducible}}  \label{proof irreducible}
In this section, we prove Theorem \ref{irreducible}.

By Lemma \ref{set of dominant monomials in terms monomials}, we have the following result.
\begin{corollary}
The modules in the second summand on the right hand side of every equation of the M-system are special. In particular, they are simple.
\end{corollary}

Therefore in order to prove Theorem \ref{irreducible}, we only need to prove that the modules in the first summand on the right hand side of every equation of the M-system are simple. We will prove that in the case of type $A_n$,
\begin{align} \label{one case of type An proof irreducible}
 \mathcal{M}^{(s-2)}_{0,\ldots, 0,\underset{i}{k_{i}+1},0,\ldots,0,\underset{j}{k_{j}},k_{j+1},\ldots,k_{n}} \otimes \mathcal{M}^{(s)}_{0,\ldots, 0,\underset{i}{k_{i}-1},0,\ldots,0,\underset{j}{k_{j}+1},k_{j+1}\ldots,k_{n}} ,
\end{align}
where $2< i+1< j \leq n$, is simple (this is $\mathcal{M}_3 \otimes \mathcal{M}_4$ which corresponds to the $4$-th line of the third column of Table \ref{dominant monomials in the M-system of type A}), and in the case of type $B_n$,
\begin{align} \label{one case of type Bn proof irreducible}
\mathcal{M}_{0,\ldots,0,k_{n-1}+1,k_{n}}^{(s-4)} \otimes \mathcal{M}_{0,\ldots,0,k_{n-1}-1,k_{n}+2}^{(s)} 
\end{align}
is simple (this is $\mathcal{M}_3 \otimes \mathcal{M}_4$ which corresponds to the $3$-rd line of the third column of Table \ref{dominant monomials in the M-system of type B}). The other cases are similar.

The following is the proof of the fact that (\ref{one case of type An proof irreducible}) is simple. Let
\begin{align*}
M_3 = M^{(s-2)}_{0,\ldots, 0,\underset{i}{k_{i}+1},0,\ldots,0,\underset{j}{k_{j}},k_{j+1},\ldots,k_{n}}, \quad  M_4 = M^{(s)}_{0,\ldots, 0,\underset{i}{k_{i}-1},0,\ldots,0,\underset{j}{k_{j}+1},k_{j+1},\ldots,k_{n}}.
\end{align*}

By Lemma \ref{set of dominant monomials in terms monomials}, the dominant monomials in $\chi_q(M_3) \chi_q(M_4)$ are
\begin{align*}
M_r = M_3 M_4 \prod_{0\leq j\leq r}A^{-1}_{i,s+i+2k_{i}-2j-4}, \quad -1\leq r\leq k_{i}-2.
\end{align*}

We need to show that $\chi_{q}(M_{r})\nsubseteq \chi_q(M_3) \chi_q(M_4)$ for $0 \leq r \leq k_{i}-2$. We will prove the case of $r=0$. The other cases are similar. By Corollary \ref{Uqsl_2 arguments}, the monomial $n_{1} = M_3 M_4 A^{-2}_{i,s+i+2k_{i}-4}$ is in $\chi_{q}(M_{0})$.

Suppose that $n_{1}\in \chi_{q}(M_3)\chi_{q}(M_4)$. Then $n_1=m_{1}m_{2}$, where $m_1 \in \chi_{q}(M_3)$, $m_2 \in \chi_{q}(M_4)$. Since $n_1=M_3 M_4 A^{-2}_{i,s+i+2k_{i}-4}$, by the expressions $M_3$ and $M_4$ we must have
\begin{align*}
m_1 = M_3 A^{-1}_{i,s+i+2k_{i}-4}, \quad m_{2} = M_4 A^{-1}_{i,s+i+2k_{i}-4}.
\end{align*}
But by the Frenkel-Mukhin algorithm, $M_3 A^{-1}_{i,s+i+2k_{i}-4}$ is not in $\chi_{q}(M_3)$. This is a contradiction. Therefore $n_{1} \not\in \chi_{q}(M_3)\chi_{q}(M_4)$ and hence $\chi_{q}(M_{0})\nsubseteq \chi_q(M_3) \chi_q(M_4)$.

The following is the proof of the fact that (\ref{one case of type Bn proof irreducible}) is simple. Let
\begin{align*}
 M_3 = M_{0,\ldots,0,k_{n-1}+1,k_{n}}^{(s-4)}, \quad M_4 = M_{0,\ldots,0,k_{n-1}-1,k_{n}+2}^{(s)}.
\end{align*}

By Lemma \ref{set of dominant monomials in terms monomials}, the dominant monomials in $\chi_q(M_3) \chi_q(M_4)$ are
\begin{align*}
M_r =M_3 M_4 \prod_{0 \leq j \leq r}A^{-1}_{n-1,s+2n+4k_{n-1}-4j-10}, \quad -1\leq r \leq k_{n-1}-2.
\end{align*}

We need to show that $\chi_{q}(M_{r})\nsubseteq \chi_q(M_3) \chi_q(M_4)$ for $0 \leq r \leq k_{n-1}-2$. We will prove the case of $r=0$. The other cases are similar. By Corollary \ref{Uqsl_2 arguments}, the monomial $n_{1} = M_3 M_4 A^{-2}_{n-1,s+2n+4k_{n-1}-10}$ is in $\chi_{q}(M_{0})$.

Suppose that $n_{1}\in \chi_{q}(M_3)\chi_{q}(M_4)$. Then $n_1=m_{1}m_{2}$, where $m_1 \in \chi_{q}(M_3)$, $m_2 \in \chi_{q}(M_4)$. Since $n_1=M_3 M_4 A^{-2}_{n-1,s+2n+4k_{n-1}-10}$, by the expressions $M_3$ and $M_4$ we must have
\begin{align*}
m_1 =M_3 A^{-1}_{n-1,s+2n+4k_{n-1}-10}, \quad m_{2} = M_4 A^{-1}_{n-1,s+2n+4k_{n-1}-10}.
\end{align*}
But by the Frenkel-Mukhin algorithm, $M_3 A^{-1}_{n-1,s+2n+4k_{n-1}-10}$ is not in $\chi_{q}(M_3)$. This is a contradiction. Hence $\chi_{q}(M_{0})\nsubseteq \chi_q(M_3) \chi_q(M_4)$.

\section*{Acknowledgement}
The authors are very grateful to the anonymous referees for their comments and suggestions that have been helpful to improve the quality of this paper. J.-R. Li would like to express his gratitude to Professor Vyjayanthi Chari for telling him a reference about the fact that all minimal affinizations are prime. J.-R. Li is supported by ERC AdG Grant 247049, the PBC Fellowship Program of Israel for Outstanding Post-Doctoral Researchers from China and India. The authors are supported by the National Natural Science Foundation of China (no. 11371177, 11401275), and the Fundamental Research Funds for the Central Universities of China (no. lzujbky-2015-78).

\section*{Appendix}
In this section, we give some examples of mutation sequences. The initial quivers in this section are the initial quivers in \cite{HL13}. The mutation sequences in this section are similar to the mutation sequences given in \cite{HL13}. In \cite{HL13}, the mutation sequences produce  Kirillov-Reshetikhin modules. In the following, the mutation sequences produce minimal affinizations. A box at a vertex indicates that a mutation has been performed at the vertex. Figure \ref{mutation sequence 11} and Figure \ref{mutation sequence 21211} are examples of mutation sequences of type $A_3$. Figure \ref{mutation sequence B11} and Figure \ref{mutation sequence B33} are examples of mutation sequences of type $B_2$.

\begin{figure}[H]
\resizebox{.6\width}{.6\height}{
\begin{minipage}[t]{.10\linewidth}
\small\small\begin{xy}
(25,50)*+{t_{0,1,0}^{(-1)}}="a";%
(10,40)*+{t_{1,0,0}^{(-1)}}="b"; (40,40)*+{t_{0,0,1}^{(-3)}}="c";%
(25,30)*+{t_{0,2,0}^{(-3)}}="d";%
(10,20)*+{t_{2,0,0}^{(-3)}}="e"; (40,20)*+{t_{0,0,2}^{(-5)}}="f";%
(25,10)*+{t_{0,3,0}^{(-5)}}="g";%
(10,0)*+{t_{3,0,0}^{(-5)}}="h"; (40,0)*+{t_{0,0,3}^{(-7)}}="i";%
(10,-15)*+{\vdots}="j"; (25,-15)*+{\vdots}="k";(40,-15)*+{\vdots}="m";%
(25,-25)*+{(a)}="n";%
{\ar "a";"b"};{\ar "a";"c"};%
{\ar "b";"d"};{\ar "c";"d"};%
{\ar "d";"a"};{\ar "d";"e"};{\ar "d";"f"};%
{\ar "e";"b"};{\ar "e";"g"};{\ar "f";"c"};{\ar "f";"g"};%
{\ar "g";"d"};{\ar "g";"h"};{\ar "g";"i"};%
{\ar "h";"e"};{\ar "i";"f"};%
{\ar "j";"h"};{\ar "k";"g"};{\ar "m";"i"};%
{\ar "h";"k"};{\ar "i";"k"};%
\end{xy}
\end{minipage}
\begin{minipage}[t]{.10\linewidth}
\begin{xy}
(25,50)*+{t_{0,1,0}^{(-1)}}="a";%
(10,40)*+{\fbox{$ t_{1,1,0}^{(-3)}$}} ="b"; (40,40)*+{t_{0,0,1}^{(-3)}}="c";%
(25,30)*+{t_{0,2,0}^{(-3)}}="d";%
(10,20)*+{t_{2,0,0}^{(-3)}}="e"; (40,20)*+{t_{0,0,2}^{(-5)}}="f";%
(25,10)*+{t_{0,3,0}^{(-5)}}="g";%
(10,0)*+{t_{3,0,0}^{(-5)}}="h"; (40,0)*+{t_{0,0,3}^{(-7)}}="i";%
(10,-15)*+{\vdots}="j"; (25,-15)*+{\vdots}="k";(40,-15)*+{\vdots}="m";%
(25,-25)*+{(b)}="n";%
{\ar "a";"c"};%
{\ar "b";"a"};{\ar "b";"e"};{\ar "c";"d"};%
{\ar "d";"b"};{\ar "d";"f"};%
{\ar "e";"g"};{\ar "f";"c"};{\ar "f";"g"};%
{\ar "g";"d"};{\ar "g";"h"};{\ar "g";"i"};%
{\ar "h";"e"};{\ar "i";"f"};%
{\ar "j";"h"};{\ar "k";"g"};{\ar "m";"i"};%
{\ar "h";"k"};{\ar "i";"k"};%
\end{xy}
\end{minipage}
\begin{minipage}[t]{.10\linewidth}
\begin{xy}
(25,50)*+{t_{0,1,0}^{(-1)}}="a";%
(10,40)*+{t_{1,1,0}^{(-3)}}="b"; (40,40)*+{t_{0,0,1}^{(-3)}}="c";%
(25,30)*+{t_{0,2,0}^{(-3)}}="d";%
(10,20)*+{\fbox{$t_{2,1,0}^{(-5)}$}}="e"; (40,20)*+{t_{0,0,2}^{(-5)}}="f";%
(25,10)*+{t_{0,3,0}^{(-5)}}="g";%
(10,0)*+{t_{3,0,0}^{(-5)}}="h"; (40,0)*+{t_{0,0,3}^{(-7)}}="i";%
(10,-15)*+{\vdots}="j"; (25,-15)*+{\vdots}="k";(40,-15)*+{\vdots}="m";%
(25,-25)*+{(c)}="n";%
{\ar "a";"c"};%
{\ar "b";"a"};{\ar "b";"g"};{\ar "c";"d"};%
{\ar "d";"b"};{\ar "d";"f"};%
{\ar "e";"b"};{\ar "e";"h"};{\ar "f";"c"};{\ar "f";"g"};%
{\ar "g";"d"};{\ar "g";"e"};{\ar "g";"i"};%
{\ar "i";"f"};%
{\ar "j";"h"};{\ar "k";"g"};{\ar "m";"i"};%
{\ar "h";"k"};{\ar "i";"k"};%
\end{xy}
\end{minipage}
}
\end{figure}

\begin{figure}[H]
\resizebox{.6\width}{.6\height}{
\begin{minipage}[t]{.10\linewidth}
\begin{xy}
(25,50)*+{t_{0,1,0}^{(-1)}}="a";%
(10,40)*+{t_{1,1,0}^{(-3)}} ="b"; (40,40)*+{t_{0,0,1}^{(-3)}}="c";%
(25,30)*+{t_{0,2,0}^{(-3)}}="d";%
(10,20)*+{t_{2,1,0}^{(-5)}}="e"; (40,20)*+{t_{0,0,2}^{(-5)}}="f";%
(25,10)*+{t_{0,3,0}^{(-5)}}="g";%
(10,0)*+{\fbox{$t_{3,1,0}^{(-7)}$}}="h"; (40,0)*+{t_{0,0,3}^{(-7)}}="i";%
(10,-15)*+{\vdots}="j"; (25,-15)*+{\vdots}="k";(40,-15)*+{\vdots}="m";%
(25,-25)*+{(d)}="n";%
{\ar "a";"c"};%
{\ar "b";"a"};{\ar "b";"g"};{\ar "c";"d"};%
{\ar "d";"b"};{\ar "d";"f"};%
{\ar "e";"b"};{\ar "f";"c"};{\ar "f";"g"};%
{\ar "g";"d"};{\ar "g";"e"};{\ar "g";"i"};%
{\ar "h";"e"};{\ar "i";"f"};%
{\ar "h";"j"};{\ar "k";"g"};{\ar "m";"i"};%
{\ar "e";"k"};{\ar "k";"h"};{\ar "i";"k"};%
\end{xy}
\end{minipage}
\begin{minipage}[t]{.10\linewidth}
\begin{xy}
(10,40)*+{\ldots}="g";
\end{xy}
\end{minipage}
\begin{minipage}[t]{.25\linewidth}
\begin{xy}
(25,50)*+{t_{0,1,0}^{(-1)}}="a";%
(10,40)*+{\fbox{$t_{1,2,0}^{(-5)}$}} ="b"; (40,40)*+{t_{0,0,1}^{(-3)}}="c";%
(25,30)*+{t_{0,2,0}^{(-3)}}="d";%
(10,20)*+{t_{2,1,0}^{(-5)}}="e"; (40,20)*+{t_{0,0,2}^{(-5)}}="f";%
(25,10)*+{t_{0,3,0}^{(-5)}}="g";%
(10,0)*+{t_{3,1,0}^{(-7)}}="h"; (40,0)*+{t_{0,0,3}^{(-7)}}="i";%
(10,-15)*+{\vdots}="j"; (25,-15)*+{\vdots}="k";(40,-15)*+{\vdots}="m";%
(25,-25)*+{(e)}="n";%
{\ar "a";"c"};%
{\ar "a";"b"};{\ar "b";"e"};{\ar "b";"d"};{\ar "c";"d"};%
{\ar "d";"a"};{\ar "d";"f"};%
{\ar "e";"a"};{\ar "f";"c"};{\ar "f";"g"};%
{\ar "g";"b"};{\ar "g";"i"};%
{\ar "h";"e"};{\ar "i";"f"};%
{\ar "h";"j"};{\ar "k";"g"};{\ar "m";"i"};%
{\ar "e";"k"};{\ar "k";"h"};{\ar "i";"k"};%
\end{xy}
\end{minipage}
\begin{minipage}[t]{.25\linewidth}
\begin{xy}
(25,50)*+{t_{0,1,0}^{(-1)}}="a";%
(10,40)*+{t_{1,2,0}^{(-5)}} ="b"; (40,40)*+{t_{0,0,1}^{(-3)}}="c";%
(25,30)*+{t_{0,2,0}^{(-3)}}="d";%
(10,20)*+{\fbox{$t_{2,2,0}^{(-7)}$}}="e"; (40,20)*+{t_{0,0,2}^{(-5)}}="f";%
(25,10)*+{t_{0,3,0}^{(-5)}}="g";%
(10,0)*+{t_{3,1,0}^{(-7)}}="h"; (40,0)*+{t_{0,0,3}^{(-7)}}="i";%
(10,-15)*+{\vdots}="j"; (25,-15)*+{\vdots}="k";(40,-15)*+{\vdots}="m";%
(25,-25)*+{(f)}="n";%
{\ar "a";"c"};%
{\ar "a";"e"};{\ar "b";"d"};{\ar "b";"k"};{\ar "c";"d"};%
{\ar "d";"a"};{\ar "d";"f"};%
{\ar "e";"b"};{\ar "f";"c"};{\ar "f";"g"};%
{\ar "g";"b"};{\ar "g";"i"};%
{\ar "e";"h"};{\ar "i";"f"};%
{\ar "h";"a"};{\ar "k";"g"};{\ar "m";"i"};%
{\ar "j";"h"};{\ar "h";"k"};{\ar "i";"k"};{\ar "k";"e"};%
\end{xy}
\end{minipage}
\begin{minipage}[t]{.10\linewidth}
\begin{xy}
(10,40)*+{\ldots}="g";
\end{xy}
\end{minipage}
}
\caption{The mutation sequence $(C_1, C_1)$.} \label{mutation sequence 11}
\end{figure}

\begin{figure}[H]
\resizebox{.6\width}{.6\height}{
\begin{minipage}[t]{.25\linewidth}
\begin{xy}
(25,50)*+{\fbox{$t_{0,1,0}^{(-3)}$}}="a";%
(10,40)*+{t_{1,0,0}^{(-1)}}="b"; (40,40)*+{t_{0,0,1}^{(-3)}}="c";%
(25,30)*+{t_{0,2,0}^{(-3)}}="d";%
(10,20)*+{t_{2,0,0}^{(-3)}}="e"; (40,20)*+{t_{0,0,2}^{(-5)}}="f";%
(25,10)*+{t_{0,3,0}^{(-5)}}="g";%
(10,0)*+{t_{3,0,0}^{(-5)}}="h"; (40,0)*+{t_{0,0,3}^{(-7)}}="i";%
(10,-15)*+{\vdots}="j"; (25,-15)*+{\vdots}="k";(40,-15)*+{\vdots}="m";%
(25,-25)*+{(a)}="n";%
{\ar "b";"a"};{\ar "c";"a"};%
{\ar "a";"d"};{\ar "d";"e"};{\ar "d";"f"};%
{\ar "e";"b"};{\ar "e";"g"};{\ar "f";"c"};{\ar "f";"g"};%
{\ar "g";"d"};{\ar "g";"h"};{\ar "g";"i"};%
{\ar "h";"e"};{\ar "i";"f"};
{\ar "j";"h"};{\ar "k";"g"};{\ar "m";"i"};%
{\ar "h";"k"};{\ar "i";"k"};%
\end{xy}
\end{minipage}
\begin{minipage}[t]{.25\linewidth}
\begin{xy}
(25,50)*+{t_{0,1,0}^{(-3)}}="a";%
(10,40)*+{t_{1,1,0}^{(-3)}} ="b"; (40,40)*+{t_{0,0,1}^{(-3)}}="c";%
(25,30)*+{\fbox{$t_{0,2,0}^{(-5)}$}}="d";%
(10,20)*+{t_{2,0,0}^{(-3)}}="e"; (40,20)*+{t_{0,0,2}^{(-5)}}="f";%
(25,10)*+{t_{0,3,0}^{(-5)}}="g";%
(10,0)*+{t_{3,0,0}^{(-5)}}="h"; (40,0)*+{t_{0,0,3}^{(-7)}}="i";%
(10,-15)*+{\vdots}="j"; (25,-15)*+{\vdots}="k";(40,-15)*+{\vdots}="m";%
(25,-25)*+{(b)}="n";%
{\ar "a";"e"};{\ar "a";"f"};%
{\ar "b";"a"};{\ar "c";"a"};%
{\ar "d";"a"};{\ar "d";"g"};%
{\ar "e";"b"};{\ar "e";"d"};{\ar "f";"c"};{\ar "f";"d"};%
{\ar "g";"h"};{\ar "g";"i"};%
{\ar "h";"e"};{\ar "i";"f"};%
{\ar "j";"h"};{\ar "k";"g"};{\ar "m";"i"};%
{\ar "h";"k"};{\ar "i";"k"};%
\end{xy}
\end{minipage}
\begin{minipage}[t]{.25\linewidth}
\begin{xy}
(25,50)*+{t_{0,1,0}^{(-3)}}="a";%
(10,40)*+{t_{1,0,0}^{(-1)}}="b"; (40,40)*+{t_{0,0,1}^{(-3)}}="c";%
(25,30)*+{t_{0,2,0}^{(-5)}}="d";%
(10,20)*+{t_{2,0,0}^{(-3)}}="e"; (40,20)*+{t_{0,0,2}^{(-5)}}="f";%
(25,10)*+{\fbox{$t_{0,3,0}^{(-7)}$}}="g";%
(10,0)*+{t_{3,0,0}^{(-5)}}="h"; (40,0)*+{t_{0,0,3}^{(-7)}}="i";%
(10,-15)*+{\vdots}="j"; (25,-15)*+{\vdots}="k";(40,-15)*+{\vdots}="m";%
(25,-25)*+{(c)}="n";%
{\ar "a";"e"};{\ar "a";"f"};%
{\ar "b";"a"};{\ar "c";"a"};%
{\ar "d";"a"};{\ar "g";"d"};{\ar "d";"h"};{\ar "d";"i"};%
{\ar "e";"b"};{\ar "e";"d"};{\ar "f";"c"};{\ar "f";"d"};%
{\ar "h";"g"};{\ar "i";"g"};%
{\ar "h";"e"};{\ar "i";"f"};%
{\ar "j";"h"};{\ar "g";"k"};{\ar "m";"i"};%
\end{xy}
\end{minipage}
\begin{minipage}[t]{.10\linewidth}
\begin{xy}
(10,40)*+{\ldots}="g";
\end{xy}
\end{minipage}
}
\end{figure}

\begin{figure}[H]
\resizebox{.6\width}{.6\height}{
\begin{minipage}[t]{.25\linewidth}
\begin{xy}
(25,50)*+{t_{0,1,0}^{(-3)}}="a";%
(10,40)*+{\fbox{$t_{1,0,0}^{(-3)}$}} ="b"; (40,40)*+{t_{0,0,1}^{(-3)}}="c";%
(25,30)*+{t_{0,2,0}^{(-5)}}="d";%
(10,20)*+{t_{2,0,0}^{(-3)}}="e"; (40,20)*+{t_{0,0,2}^{(-5)}}="f";%
(25,10)*+{t_{0,3,0}^{(-7)}}="g";%
(10,0)*+{t_{3,0,0}^{(-5)}}="h"; (40,0)*+{t_{0,0,3}^{(-70)}}="i";%
(10,-15)*+{\vdots}="j"; (25,-15)*+{\vdots}="k";(40,-15)*+{\vdots}="m";%
(25,-25)*+{(d)}="n";%
{\ar "a";"f"};%
{\ar "a";"b"};{\ar "c";"a"};%
{\ar "d";"a"};{\ar "g";"d"};{\ar "d";"h"};{\ar "d";"i"};%
{\ar "b";"e"};{\ar "e";"d"};{\ar "f";"c"};{\ar "f";"d"};%
{\ar "h";"g"};{\ar "i";"g"};%
{\ar "h";"e"};{\ar "i";"f"};%
{\ar "j";"h"};{\ar "k";"g"};{\ar "m";"i"};%
{\ar "g";"j"};{\ar "g";"m"};%
\end{xy}
\end{minipage}
\begin{minipage}[t]{.25\linewidth}
\begin{xy}
(25,50)*+{t_{0,1,0}^{(-3)}}="a";%
(10,40)*+{t_{1,0,0}^{(-3)}}="b"; (40,40)*+{t_{0,0,1}^{(-3)}}="c";%
(25,30)*+{t_{0,2,0}^{(-5)}}="d";%
(10,20)*+{\fbox{$t_{2,0,0}^{(-5)}$}}="e"; (40,20)*+{t_{0,0,2}^{(-5)}}="f";%
(25,10)*+{t_{0,3,0}^{(-7)}}="g";%
(10,0)*+{t_{3,0,0}^{(-5)}}="h"; (40,0)*+{t_{0,0,3}^{(-7)}}="i";%
(10,-15)*+{\vdots}="j"; (25,-15)*+{\vdots}="k";(40,-15)*+{\vdots}="m";%
(25,-25)*+{(e)}="n";%
{\ar "a";"f"};%
{\ar "a";"b"};{\ar "c";"a"};{\ar "b";"d"};%
{\ar "d";"a"};{\ar "g";"d"};{\ar "d";"i"};%
{\ar "e";"b"};{\ar "d";"e"};{\ar "f";"c"};{\ar "f";"d"};%
{\ar "h";"g"};{\ar "i";"g"};%
{\ar "e";"h"};{\ar "i";"f"};%
{\ar "j";"h"};{\ar "k";"g"};{\ar "m";"i"};%
{\ar "g";"j"};{\ar "g";"m"};%
\end{xy}
\end{minipage}
\begin{minipage}[t]{.25\linewidth}
\begin{xy}
(25,50)*+{t_{0,1,0}^{(-3)}}="a";%
(10,40)*+{t_{1,0,0}^{(-3)}} ="b"; (40,40)*+{t_{0,0,1}^{(-3)}}="c";%
(25,30)*+{t_{0,2,0}^{(-5)}}="d";%
(10,20)*+{t_{2,0,0}^{(-5)}}="e"; (40,20)*+{t_{0,0,2}^{(-5)}}="f";%
(25,10)*+{t_{0,3,0}^{(-7)}}="g";%
(10,0)*+{\fbox{$t_{3,0,0}^{(-7)}$}}="h"; (40,0)*+{t_{0,0,3}^{(-7)}}="i";%
(10,-15)*+{\vdots}="j"; (25,-15)*+{\vdots}="k";(40,-15)*+{\vdots}="m";%
(25,-25)*+{(f)}="n";%
{\ar "a";"f"};%
{\ar "a";"b"};{\ar "c";"a"};{\ar "b";"d"};%
{\ar "d";"a"};{\ar "g";"d"};{\ar "d";"i"};%
{\ar "e";"b"};{\ar "d";"e"};{\ar "f";"c"};{\ar "f";"d"};{\ar "e";"g"};%
{\ar "g";"h"};{\ar "i";"g"};%
{\ar "h";"e"};{\ar "i";"f"};%
{\ar "h";"j"};{\ar "k";"g"};{\ar "m";"i"};%
{\ar "g";"m"};%
\end{xy}
\end{minipage}
\begin{minipage}[t]{.10\linewidth}
\begin{xy}
(10,40)*+{\ldots}="g";
\end{xy}
\end{minipage}
}
\end{figure}

\begin{figure}[H]
\resizebox{.6\width}{.6\height}{
\begin{minipage}[t]{.25\linewidth}
\begin{xy}
(25,50)*+{\fbox{$t_{0,1,1}^{(-5)}$}}="a";%
(10,40)*+{t_{1,0,0}^{(-3)}}="b"; (40,40)*+{t_{0,0,1}^{(-3)}}="c";%
(25,30)*+{t_{0,2,0}^{(-5)}}="d";%
(10,20)*+{t_{2,0,0}^{(-5)}}="e"; (40,20)*+{t_{0,0,2}^{(-5)}}="f";%
(25,10)*+{t_{0,3,0}^{(-7)}}="g";%
(10,0)*+{t_{3,0,0}^{(-7)}}="h"; (40,0)*+{t_{0,0,3}^{(-7)}}="i";%
(10,-15)*+{\vdots}="j"; (25,-15)*+{\vdots}="k";(40,-15)*+{\vdots}="m";%
(25,-25)*+{(g)}="n";%
{\ar "b";"a"};{\ar "a";"c"};{\ar "c";"b"};%
{\ar "a";"d"};{\ar "g";"d"};{\ar "d";"i"};%
{\ar "e";"b"};{\ar "d";"e"};{\ar "f";"a"};{\ar "e";"g"};%
{\ar "g";"h"};{\ar "i";"g"};%
{\ar "h";"e"};{\ar "i";"f"};%
{\ar "j";"h"};{\ar "g";"k"};{\ar "m";"i"};%
{\ar "h";"k"};{\ar "g";"m"};%
\end{xy}
\end{minipage}
\begin{minipage}[t]{.25\linewidth}
\begin{xy}
(25,50)*+{t_{0,1,1}^{(-5)}}="a";%
(10,40)*+{t_{1,0,0}^{(-3)}} ="b";
(40,40)*+{t_{0,0,1}^{(-3)}}="c";%
(25,30)*+{\fbox{$t_{0,2,1}^{(-7)}$}}="d";%
(10,20)*+{t_{2,0,0}^{(-5)}}="e"; (40,20)*+{t_{0,0,2}^{(-5)}}="f";%
(25,10)*+{t_{0,3,0}^{(-5)}}="g";%
(10,0)*+{t_{3,0,0}^{(-7)}}="h"; (40,0)*+{t_{0,0,3}^{(-7)}}="i";%
(10,-15)*+{\vdots}="j"; (25,-15)*+{\vdots}="k";(40,-15)*+{\vdots}="m";%
(25,-25)*+{(h)}="n";%
{\ar "a";"i"};{\ar "a";"e"};%
{\ar "b";"a"};{\ar "a";"c"};{\ar "c";"b"};%
{\ar "d";"a"};{\ar "d";"g"};{\ar "i";"d"};%
{\ar "e";"b"};{\ar "e";"d"};{\ar "f";"a"};%
{\ar "g";"h"};%
{\ar "h";"e"};{\ar "i";"f"};%
{\ar "j";"h"};{\ar "g";"k"};{\ar "m";"i"};%
{\ar "h";"k"};{\ar "g";"m"};%
\end{xy}
\end{minipage}
\begin{minipage}[t]{.25\linewidth}
\begin{xy}
(25,50)*+{t_{0,1,1}^{(-5)}}="a";%
(10,40)*+{t_{1,0,0}^{(-3)}}="b"; (40,40)*+{t_{0,0,1}^{(-3)}}="c";%
(25,30)*+{t_{0,2,1}^{(-7)}}="d";%
(10,20)*+{t_{2,0,0}^{(-5)}}="e"; (40,20)*+{t_{0,0,2}^{(-5)}}="f";%
(25,10)*+{\fbox{$t_{0,3,1}^{(-9)}$}}="g";%
(10,0)*+{t_{3,0,0}^{(-7)}}="h"; (40,0)*+{t_{0,0,3}^{(-7)}}="i";%
(10,-15)*+{\vdots}="j"; (25,-15)*+{\vdots}="k";(40,-15)*+{\vdots}="m";%
(25,-25)*+{(i)}="n";%
{\ar "a";"i"};{\ar "a";"e"};%
{\ar "b";"a"};{\ar "a";"c"};{\ar "c";"b"};%
{\ar "d";"a"};{\ar "d";"h"};{\ar "g";"d"};{\ar "i";"d"};%
{\ar "e";"b"};{\ar "e";"d"};{\ar "f";"a"};%
{\ar "h";"g"};%
{\ar "h";"e"};{\ar "i";"f"};%
{\ar "j";"h"};{\ar "g";"k"};{\ar "m";"i"};%
{\ar "d";"m"};{\ar "m";"g"};%
\end{xy}
\end{minipage}
\begin{minipage}[t]{.10\linewidth}
\begin{xy}
(10,40)*+{\ldots}="g";
\end{xy}
\end{minipage}
}
\end{figure}

\begin{figure}[H]
\resizebox{.6\width}{.6\height}{
\begin{minipage}[t]{.25\linewidth}
\begin{xy}
(25,50)*+{t_{0,1,1}^{(-5)}}="a";%
(10,40)*+{\fbox{$t_{1,0,1}^{(-5)}$}} ="b"; (40,40)*+{t_{0,0,1}^{(-3)}}="c";%
(25,30)*+{t_{0,2,1}^{(-7)}}="d";%
(10,20)*+{t_{2,0,0}^{(-5)}}="e"; (40,20)*+{t_{0,0,2}^{(-5)}}="f";%
(25,10)*+{t_{0,3,1}^{(-9)}}="g";%
(10,0)*+{t_{3,0,0}^{(-7)}}="h"; (40,0)*+{t_{0,0,3}^{(-7)}}="i";%
(10,-15)*+{\vdots}="j"; (25,-15)*+{\vdots}="k";(40,-15)*+{\vdots}="m";(40,-22)*+{}="p";%
(25,-25)*+{(j)}="n";%
{\ar "a";"i"};%
{\ar "a";"b"};{\ar "b";"e"};{\ar "b";"c"};%
{\ar "d";"a"};{\ar "d";"h"};{\ar "g";"d"};{\ar "i";"d"};%
{\ar "e";"d"};{\ar "f";"a"};%
{\ar "h";"g"};%
{\ar "h";"e"};{\ar "i";"f"};%
{\ar "j";"h"};{\ar "k";"g"};{\ar "m";"i"};%
{\ar "d";"m"};{\ar "g";"j"};{\ar "m";"g"};{\ar "g";"p"};%
\end{xy}
\end{minipage}
\begin{minipage}[t]{.25\linewidth}
\begin{xy}
(25,50)*+{t_{0,1,1}^{(-5)}}="a";%
(10,40)*+{t_{1,0,1}^{(-5)}}="b"; (40,40)*+{t_{0,0,1}^{(-3)}}="c";%
(25,30)*+{t_{0,2,1}^{(-7)}}="d";%
(10,20)*+{\fbox{$t_{2,0,1}^{(-7)}$}}="e"; (40,20)*+{t_{0,0,2}^{(-5)}}="f";%
(25,10)*+{t_{0,3,1}^{(-9)}}="g";%
(10,0)*+{t_{3,0,0}^{(-7)}}="h"; (40,0)*+{t_{0,0,3}^{(-7)}}="i";%
(10,-15)*+{\vdots}="j"; (25,-15)*+{\vdots}="k";(40,-15)*+{\vdots}="m";(40,-22)*+{}="p";%
(25,-25)*+{(k)}="n";%
{\ar "a";"i"};%
{\ar "a";"b"};{\ar "e";"b"};{\ar "b";"c"};{\ar "b";"d"};%
{\ar "d";"a"};{\ar "g";"d"};{\ar "i";"d"};%
{\ar "d";"e"};{\ar "e";"h"};{\ar "f";"a"};%
{\ar "h";"g"};%
{\ar "i";"f"};%
{\ar "j";"h"};{\ar "k";"g"};{\ar "m";"i"};%
{\ar "d";"m"};{\ar "g";"j"};{\ar "m";"g"};{\ar "g";"p"};%
\end{xy}
\end{minipage}
\begin{minipage}[t]{.25\linewidth}
\begin{xy}
(25,50)*+{t_{0,1,1}^{(-5)}}="a";%
(10,40)*+{t_{1,0,1}^{(-5)}} ="b"; (40,40)*+{t_{0,0,1}^{(-3)}}="c";%
(25,30)*+{t_{0,2,1}^{(-7)}}="d";%
(10,20)*+{t_{2,0,1}^{(-7)}}="e"; (40,20)*+{t_{0,0,2}^{(-5)}}="f";%
(25,10)*+{t_{0,3,1}^{(-9)}}="g";%
(10,0)*+{\fbox{$t_{3,0,1}^{(-9)}$}}="h"; (40,0)*+{t_{0,0,3}^{(-7)}}="i";%
(10,-15)*+{\vdots}="j"; (25,-15)*+{\vdots}="k";(40,-15)*+{\vdots}="m";(40,-22)*+{}="p";%
(25,-25)*+{(l)}="n";%
{\ar "a";"i"};%
{\ar "a";"b"};{\ar "e";"b"};{\ar "b";"c"};{\ar "b";"d"};%
{\ar "d";"a"};{\ar "g";"d"};{\ar "i";"d"};%
{\ar "d";"e"};{\ar "e";"g"};{\ar "h";"e"};{\ar "f";"a"};%
{\ar "g";"h"};%
{\ar "i";"f"};%
{\ar "j";"h"};{\ar "k";"g"};{\ar "m";"i"};%
{\ar "d";"m"};{\ar "m";"g"};{\ar "g";"p"};%
\end{xy}
\end{minipage}
\begin{minipage}[t]{.10\linewidth}
\begin{xy}
(10,40)*+{\ldots}="g";
\end{xy}
\end{minipage}
}
\end{figure}

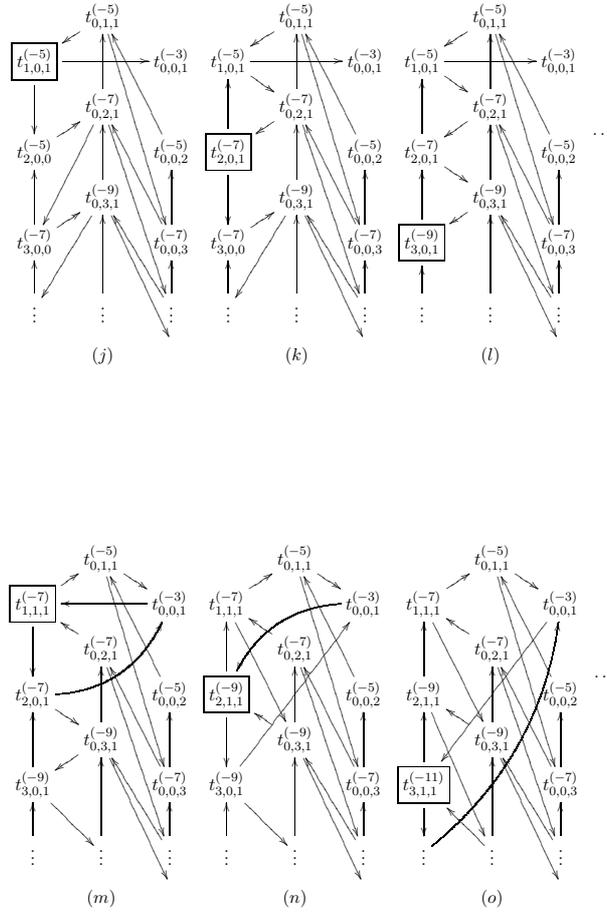
\begin{figure}[H]
\resizebox{.6\width}{.6\height}{
\begin{minipage}[t]{.25\linewidth}
\begin{xy}
(25,50)*+{t_{0,1,1}^{(-5)}}="a";%
(10,40)*+{\fbox{$t_{1,1,1}^{(-7)}$}}="b"; (40,40)*+{t_{0,0,1}^{(-3)}}="c";%
(25,30)*+{t_{0,2,1}^{(-7)}}="d";%
(10,20)*+{t_{2,0,1}^{(-7)}}="e"; (40,20)*+{t_{0,0,2}^{(-5)}}="f";%
(25,10)*+{t_{0,3,1}^{(-9)}}="g";%
(10,0)*+{t_{3,0,1}^{(-9)}}="h"; (40,0)*+{t_{0,0,3}^{(-7)}}="i";%
(10,-15)*+{\vdots}="j"; (25,-15)*+{\vdots}="k";(40,-15)*+{\vdots}="m";(40,-22)*+{}="p";%
(25,-25)*+{(m)}="n";%
{\ar "a";"i"};{\ar "a";"c"};%
{\ar "b";"a"};{\ar "b";"e"};{\ar "c";"b"};{\ar "d";"b"};%
{\ar "g";"d"};{\ar "i";"d"};%
{\ar "e";"g"};{\ar@/_1.6pc/"e";"c"};{\ar "h";"e"};{\ar "f";"a"};%
{\ar "g";"h"};%
{\ar "i";"f"};%
{\ar "j";"h"};{\ar "k";"g"};{\ar "m";"i"};%
{\ar "d";"m"};{\ar "h";"k"};{\ar "m";"g"};{\ar "g";"p"};%
\end{xy}
\end{minipage}
\begin{minipage}[t]{.25\linewidth}
\begin{xy}
(25,50)*+{t_{0,1,1}^{(-5)}}="a";%
(10,40)*+{t_{1,1,1}^{(-7)}} ="b"; (40,40)*+{t_{0,0,1}^{(-3)}}="c";%
(25,30)*+{t_{0,2,1}^{(-7)}}="d";%
(10,20)*+{\fbox{$t_{2,1,1}^{(-9)}$}}="e"; (40,20)*+{t_{0,0,2}^{(-5)}}="f";%
(25,10)*+{t_{0,3,1}^{(-9)}}="g";%
(10,0)*+{t_{3,0,1}^{(-9)}}="h"; (40,0)*+{t_{0,0,3}^{(-7)}}="i";%
(10,-15)*+{\vdots}="j"; (25,-15)*+{\vdots}="k";(40,-15)*+{\vdots}="m";(40,-22)*+{}="p";%
(25,-25)*+{(n)}="n";%
{\ar "a";"i"};{\ar "a";"c"};%
{\ar "b";"a"};{\ar "b";"g"};{\ar "e";"b"};{\ar@/_1.6pc/"c";"e"};{\ar "d";"b"};{\ar "e";"h"};%
{\ar "g";"d"};{\ar "g";"e"};{\ar "i";"d"};%
{\ar "f";"a"};%
{\ar "h";"c"};%
{\ar "i";"f"};%
{\ar "j";"h"};{\ar "k";"g"};{\ar "m";"i"};%
{\ar "d";"m"};{\ar "h";"k"};{\ar "m";"g"};{\ar "g";"p"};%
\end{xy}
\end{minipage}
\begin{minipage}[t]{.25\linewidth}
\begin{xy}
(25,50)*+{t_{0,1,1}^{(-5)}}="a";%
(10,40)*+{t_{1,1,1}^{(-7)}}="b"; (40,40)*+{t_{0,0,1}^{(-3)}}="c";%
(25,30)*+{t_{0,2,1}^{(-7)}}="d";%
(10,20)*+{t_{2,1,1}^{(-9)}}="e"; (40,20)*+{t_{0,0,2}^{(-5)}}="f";%
(25,10)*+{t_{0,3,1}^{(-9)}}="g";%
(10,0)*+{\fbox{$t_{3,1,1}^{(-11)}$}}="h"; (40,0)*+{t_{0,0,3}^{(-7)}}="i";%
(10,-15)*+{\vdots}="j"; (25,-15)*+{\vdots}="k";(40,-15)*+{\vdots}="m";(40,-22)*+{}="p";%
(25,-25)*+{(o)}="n";%
{\ar "a";"i"};{\ar "a";"c"};%
{\ar "b";"a"};{\ar "b";"g"};{\ar "e";"b"};{\ar "c";"h"};{\ar "d";"b"};{\ar "h";"e"};%
{\ar "g";"d"};{\ar "g";"e"};{\ar "i";"d"};%
{\ar "f";"a"};%
{\ar "i";"f"};%
{\ar "h";"j"};{\ar "k";"g"};{\ar "m";"i"};%
{\ar "d";"m"};{\ar "k";"h"};{\ar@/_1.6pc/"j";"c"};{\ar "e";"k"};{\ar "m";"g"};{\ar "g";"p"};%
\end{xy}
\end{minipage}
\begin{minipage}[t]{.10\linewidth}
\begin{xy}
(10,40)*+{\ldots}="g";
\end{xy}
\end{minipage}
}
\caption{The mutation sequence $(C_2, C_1, C_2, C_1, C_1)$.} \label{mutation sequence 21211}
\end{figure}

\begin{figure}[H]
\resizebox{.6\width}{.6\height}{
\begin{minipage}[t]{.25\linewidth}
\begin{xy}
(15,60)*+{t_{0,1}^{(-1)}}="1";%
(15,50)*+{t_{0,2}^{(-3)}}="2"; (30,50)*+{t_{1,0}^{(-1)}}="3";%
(0,40)*+{t_{1,0}^{(-3)}}="4"; (15,40)*+{t_{0,3}^{(-5)}}="5";%
(15,30)*+{t_{0,4}^{(-7)}}="6"; (30,30)*+{t_{2,0}^{(-5)}}="7";%
(0,20)*+{t_{2,0}^{(-7)}}="8"; (15,20)*+{t_{0,5}^{(-9)}}="9";%
(15,10)*+{t_{0,6}^{(-11)}}="10"; (30,10)*+{t_{3,0}^{(-9)}}="11";%
(0,0)*+{t_{3,0}^{(-11)}}="12";(15,0)*+{t_{0,7}^{(-13)}}="13";%
(0,-15)*+{\vdots}="14";(15,-15)*+{\vdots}="15";(30,-15)*+{\vdots}="16";%
(15,-25)*+{(a)}="n";%
{\ar "2";"1"};{\ar "5";"2"};{\ar "6";"5"};{\ar "9";"6"};{\ar "10";"9"};{\ar "13";"10"};{\ar "15";"13"};%
{\ar "12";"8"};{\ar "8";"4"};{\ar "11";"7"};{\ar "7";"3"};{\ar "16";"11"};{\ar "14";"12"};%
{\ar "2";"4"};{\ar "4";"6"};{\ar "6";"8"};{\ar "8";"10"};{\ar "10";"12"};{\ar "12";"15"};%
{\ar "1";"3"};{\ar "3";"5"};{\ar "5";"7"};{\ar "7";"9"};{\ar "9";"11"};{\ar "11";"13"};{\ar "13";"16"};%
\end{xy}
\end{minipage}
\begin{minipage}[t]{.25\linewidth}
\begin{xy}
(16,60)*+{t_{0,1}^{(-1)}}="1";%
(15,50)*+{t_{0,2}^{(-3)}}="2"; (30,50)*+{\fbox{$t_{1,1}^{(-5)}$}}="3";%
(0,40)*+{t_{1,0}^{(-3)}}="4"; (15,40)*+{t_{0,3}^{(-5)}}="5";%
(15,30)*+{t_{0,4}^{(-7)}}="6"; (30,30)*+{t_{2,0}^{(-5)}}="7";%
(0,20)*+{t_{2,0}^{(-7)}}="8"; (15,20)*+{t_{0,5}^{(-9)}}="9";%
(15,10)*+{t_{0,6}^{(-11)}}="10"; (30,10)*+{t_{3,0}^{(-9)}}="11";%
(0,0)*+{t_{3,0}^{(-11)}}="12";(15,0)*+{t_{0,7}^{(-13)}}="13";%
(0,-15)*+{\vdots}="14";(15,-15)*+{\vdots}="15";(30,-15)*+{\vdots}="16";%
(15,-25)*+{(b)}="n";%
{\ar "2";"1"};{\ar "5";"2"};{\ar "6";"5"};{\ar "9";"6"};{\ar "10";"9"};{\ar "13";"10"};{\ar "15";"13"};%
{\ar "12";"8"};{\ar "8";"4"};{\ar "11";"7"};{\ar "3";"7"};{\ar "16";"11"};{\ar "14";"12"};%
{\ar "2";"4"};{\ar "4";"6"};{\ar "6";"8"};{\ar "8";"10"};{\ar "10";"12"};{\ar "12";"15"};%
{\ar "3";"1"};{\ar "5";"3"};{\ar "7";"9"};{\ar "9";"11"};{\ar "11";"13"};{\ar "13";"16"};%
{\ar@/_2pc/"1";"5"};%
\end{xy}
\end{minipage}
\begin{minipage}[t]{.25\linewidth}
\begin{xy}
(15,60)*+{t_{0,1}^{(-1)}}="1";%
(15,50)*+{t_{0,2}^{(-3)}}="2"; (30,50)*+{t_{1,1}^{(-5)}}="3";%
(0,40)*+{t_{1,0}^{(-3)}}="4"; (15,40)*+{t_{0,3}^{(-5)}}="5";%
(15,30)*+{t_{0,4}^{(-7)}}="6"; (30,30)*+{\fbox{$t_{2,1}^{(-9)}$}}="7";%
(0,20)*+{t_{2,0}^{(-7)}}="8"; (15,20)*+{t_{0,5}^{(-9)}}="9";%
(15,10)*+{t_{0,6}^{(-11)}}="10"; (30,10)*+{t_{3,0}^{(-9)}}="11";%
(0,0)*+{t_{3,0}^{(-11)}}="12";(15,0)*+{t_{0,7}^{(-13)}}="13";%
(0,-15)*+{\vdots}="14";(15,-15)*+{\vdots}="15";(30,-15)*+{\vdots}="16";%
(15,-25)*+{(c)}="n";%
{\ar "2";"1"};{\ar "5";"2"};{\ar "6";"5"};{\ar "9";"6"};{\ar "10";"9"};{\ar "13";"10"};{\ar "15";"13"};%
{\ar "12";"8"};{\ar "8";"4"};{\ar "7";"11"};{\ar "7";"3"};{\ar "16";"11"};{\ar "14";"12"};%
{\ar "2";"4"};{\ar "4";"6"};{\ar "6";"8"};{\ar "8";"10"};{\ar "10";"12"};{\ar "12";"15"};%
{\ar "3";"1"};{\ar "5";"3"};{\ar "9";"7"};{\ar "11";"13"};{\ar "13";"16"};%
{\ar@/_2pc/"1";"5"};{\ar "3";"9"};%
\end{xy}
\end{minipage}
}
\end{figure}

\begin{figure}[H]
\resizebox{.6\width}{.6\height}{
\begin{minipage}[t]{.25\linewidth}
\begin{xy}
(15,60)*+{t_{0,1}^{(-1)}}="1";%
(15,50)*+{t_{0,2}^{(-3)}}="2"; (30,50)*+{t_{1,1}^{(-5)}}="3";%
(0,40)*+{t_{1,0}^{(-3)}}="4"; (15,40)*+{t_{0,3}^{(-5)}}="5";%
(15,30)*+{t_{0,4}^{(-7)}}="6"; (30,30)*+{t_{2,1}^{(-9)}}="7";%
(0,20)*+{t_{2,0}^{(-7)}}="8"; (15,20)*+{t_{0,5}^{(-9)}}="9";%
(15,10)*+{t_{0,6}^{(-11)}}="10"; (30,10)*+{\fbox{$t_{3,1}^{(-13)}$}}="11";%
(0,0)*+{t_{3,0}^{(-11)}}="12";(15,0)*+{t_{0,7}^{(-13)}}="13";%
(0,-15)*+{\vdots}="14";(15,-15)*+{\vdots}="15";(30,-15)*+{\vdots}="16";%
(15,-25)*+{(d)}="n";%
{\ar "2";"1"};{\ar "5";"2"};{\ar "6";"5"};{\ar "9";"6"};{\ar "10";"9"};{\ar "13";"10"};{\ar "15";"13"};%
{\ar "12";"8"};{\ar "8";"4"};{\ar "11";"7"};{\ar "7";"3"};{\ar "11";"16"};{\ar "14";"12"};%
{\ar "2";"4"};{\ar "4";"6"};{\ar "6";"8"};{\ar "8";"10"};{\ar "10";"12"};{\ar "12";"15"};%
{\ar "3";"1"};{\ar "5";"3"};{\ar "9";"7"};{\ar "13";"11"};%
{\ar@/_2pc/"1";"5"};{\ar "3";"9"};{\ar "7";"13"};%
\end{xy}
\end{minipage}
\begin{minipage}[t]{.25\linewidth}
\begin{xy}
(10,70)*+{\ldots}="g";
\end{xy}
\end{minipage}
\begin{minipage}[t]{.25\linewidth}
\begin{xy}
(15,60)*+{t_{0,1}^{(-1)}}="1";%
(15,50)*+{t_{0,2}^{(-3)}}="2"; (30,50)*+{\fbox{$t_{1,3}^{(-9)}$}}="3";%
(0,40)*+{t_{1,0}^{(-3)}}="4"; (15,40)*+{t_{0,3}^{(-5)}}="5";%
(15,30)*+{t_{0,4}^{(-7)}}="6"; (30,30)*+{t_{2,1}^{(-9)}}="7";%
(0,20)*+{t_{2,0}^{(-7)}}="8"; (15,20)*+{t_{0,5}^{(-9)}}="9";%
(15,10)*+{t_{0,6}^{(-11)}}="10"; (30,10)*+{t_{3,1}^{(-13)}}="11";%
(0,0)*+{t_{3,0}^{(-11)}}="12";(15,0)*+{t_{0,7}^{(-13)}}="13";%
(0,-15)*+{\vdots}="14";(15,-15)*+{\vdots}="15";(30,-15)*+{\vdots}="16";(15,-22)*+{}="17";%
(15,-25)*+{(e)}="n";%
{\ar "2";"1"};{\ar "5";"2"};{\ar "6";"5"};{\ar "9";"6"};{\ar "10";"9"};{\ar "13";"10"};{\ar "15";"13"};%
{\ar "12";"8"};{\ar "8";"4"};{\ar "11";"7"};{\ar "3";"7"};{\ar "16";"11"};{\ar "14";"12"};%
{\ar "2";"4"};{\ar "4";"6"};{\ar "6";"8"};{\ar "8";"10"};{\ar "10";"12"};{\ar "12";"15"};%
{\ar "1";"3"};{\ar "3";"5"};{\ar "13";"11"};%
{\ar "9";"3"};{\ar "7";"13"};{\ar "7";"1"};{\ar@/_2pc/"5";"9"};{\ar "11";"17"};%
\end{xy}
\end{minipage}
\begin{minipage}[t]{.25\linewidth}
\begin{xy}
(15,60)*+{t_{0,1}^{(-1)}}="1";%
(15,50)*+{t_{0,2}^{(-3)}}="2"; (30,50)*+{t_{1,3}^{(-9)}}="3";%
(0,40)*+{t_{1,0}^{(-3)}}="4"; (15,40)*+{t_{0,3}^{(-5)}}="5";%
(15,30)*+{t_{0,4}^{(-7)}}="6"; (30,30)*+{\fbox{$t_{2,3}^{(-13)}$}}="7";%
(0,20)*+{t_{2,0}^{(-7)}}="8"; (15,20)*+{t_{0,5}^{(-9)}}="9";%
(15,10)*+{t_{0,6}^{(-11)}}="10"; (30,10)*+{t_{3,1}^{(-13)}}="11";%
(0,0)*+{t_{3,0}^{(-11)}}="12";(15,0)*+{t_{0,7}^{(-13)}}="13";%
(0,-15)*+{\vdots}="14";(15,-15)*+{\vdots}="15";(30,-15)*+{\vdots}="16";(15,-22)*+{}="17";%
(15,-25)*+{(f)}="n";%
{\ar "2";"1"};{\ar "5";"2"};{\ar "6";"5"};{\ar "9";"6"};{\ar "10";"9"};{\ar "13";"10"};{\ar "15";"13"};%
{\ar "12";"8"};{\ar "8";"4"};{\ar "7";"11"};{\ar "7";"3"};{\ar "16";"11"};{\ar "14";"12"};%
{\ar "2";"4"};{\ar "4";"6"};{\ar "6";"8"};{\ar "8";"10"};{\ar "10";"12"};{\ar "12";"15"};%
{\ar "3";"5"};%
{\ar "9";"3"};{\ar "13";"7"};{\ar "1";"7"};{\ar@/_2pc/"5";"9"};{\ar "3";"13"};{\ar "11";"1"};{\ar "11";"17"};%
\end{xy}
\end{minipage}
\begin{minipage}[t]{.25\linewidth}
\begin{xy}
(10,70)*+{\ldots}="g";
\end{xy}
\end{minipage}
}
\caption{The mutation sequence $(C_1, C_1)$.} \label{mutation sequence B11}
\end{figure}

\begin{figure}[H]
\resizebox{.6\width}{.6\height}{
\begin{minipage}[t]{.25\linewidth}
\begin{xy}
(15,60)*+{t_{0,1}^{(-1)}}="1";%
(15,50)*+{t_{0,2}^{(-3)}}="2"; (30,50)*+{t_{1,0}^{(-1)}}="3";%
(0,40)*+{\fbox{$t_{1,2}^{(-7)}$}}="4"; (15,40)*+{t_{0,3}^{(-5)}}="5";%
(15,30)*+{t_{0,4}^{(-7)}}="6"; (30,30)*+{t_{2,0}^{(-5)}}="7";%
(0,20)*+{(t_{2,0}^{(-7)})}="8"; (15,20)*+{t_{0,5}^{(-9)}}="9";%
(15,10)*+{t_{0,6}^{-11}}="10"; (30,10)*+{t_{3,0}^{(-9)}}="11";%
(0,0)*+{t_{3,0}^{(-11)}}="12";(15,0)*+{t_{0,7}^{(-13)}}="13";%
(0,-15)*+{\vdots}="14";(15,-15)*+{\vdots}="15";(30,-15)*+{\vdots}="16";%
(15,-25)*+{(a)}="n";%
{\ar "2";"1"};{\ar "5";"2"};{\ar "6";"5"};{\ar "9";"6"};{\ar "10";"9"};{\ar "13";"10"};{\ar "15";"13"};%
{\ar "12";"8"};{\ar "4";"8"};{\ar "11";"7"};{\ar "7";"3"};{\ar "16";"11"};%
{\ar "4";"2"};{\ar "6";"4"};{\ar "8";"10"};{\ar "10";"12"};{\ar "12";"15"};{\ar "14";"12"};%
{\ar "1";"3"};{\ar "3";"5"};{\ar "5";"7"};{\ar "7";"9"};{\ar "9";"11"};{\ar "11";"13"};{\ar "13";"16"};%
{\ar@/_1.5pc/"2";"6"};%
\end{xy}
\end{minipage}
\begin{xy}
(15,60)*+{t_{0,1}^{(-1)}}="1";%
(15,50)*+{t_{0,2}^{(-3)}}="2"; (30,50)*+{t_{1,0}^{(-1)}}="3";%
(0,40)*+{t_{1,2}^{(-7)}}="4"; (15,40)*+{t_{0,3}^{(-5)}}="5";%
(15,30)*+{t_{0,4}^{(-7)}}="6"; (30,30)*+{t_{2,0}^{(-5)}}="7";%
(0,20)*+{\fbox{$t_{2,2}^{(-11)}$}}="8"; (15,20)*+{t_{0,5}^{(-9)}}="9";%
(15,10)*+{t_{0,6}^{(-11)}}="10"; (30,10)*+{t_{3,0}^{(-9)}}="11";%
(0,0)*+{t_{3,0}^{(-11)}}="12";(15,0)*+{t_{0,7}^{(-13)}}="13";%
(0,-15)*+{\vdots}="14";(15,-15)*+{\vdots}="15";(30,-15)*+{\vdots}="16";%
(15,-25)*+{(b)}="n";%
{\ar "2";"1"};{\ar "5";"2"};{\ar "6";"5"};{\ar "9";"6"};{\ar "10";"9"};{\ar "13";"10"};{\ar "15";"13"};%
{\ar "8";"12"};{\ar "8";"4"};{\ar "11";"7"};{\ar "7";"3"};{\ar "16";"11"};%
{\ar "4";"2"};{\ar "6";"4"};{\ar "10";"8"};{\ar "12";"15"};{\ar "14";"12"};%
{\ar "1";"3"};{\ar "3";"5"};{\ar "5";"7"};{\ar "7";"9"};{\ar "9";"11"};{\ar "11";"13"};{\ar "13";"16"};%
{\ar@/_1.5pc/"2";"6"};{\ar "4";"10"};%
\end{xy}
\begin{minipage}[t]{.25\linewidth}
\begin{xy}
(15,60)*+{t_{0,1}^{(-1)}}="1";%
(15,50)*+{t_{0,2}^{(-3)}}="2"; (30,50)*+{t_{1,0}^{(-1)}}="3";%
(0,40)*+{t_{1,2}^{(-7)}}="4"; (15,40)*+{t_{0,3}^{(-5)}}="5";%
(15,30)*+{t_{0,4}^{(-7)}}="6"; (30,30)*+{t_{2,0}^{(-5)}}="7";%
(0,20)*+{t_{2,2}^{(-11)}}="8"; (15,20)*+{t_{0,5}^{(-9)}}="9";%
(15,10)*+{t_{0,6}^{(-11)}}="10"; (30,10)*+{t_{3,0}^{(-9)}}="11";%
(0,0)*+{\fbox{$t_{3,2}^{(-15)}$}}="12";(15,0)*+{t_{0,7}^{(-13)}}="13";%
(0,-15)*+{\vdots}="14";(15,-15)*+{\vdots}="15";(30,-15)*+{\vdots}="16";%
(15,-25)*+{(c)}="n";%
{\ar "2";"1"};{\ar "5";"2"};{\ar "6";"5"};{\ar "9";"6"};{\ar "10";"9"};{\ar "13";"10"};{\ar "15";"13"};%
{\ar "12";"8"};{\ar "8";"4"};{\ar "11";"7"};{\ar "7";"3"};{\ar "16";"11"};%
{\ar "4";"2"};{\ar "6";"4"};{\ar "10";"8"};{\ar "15";"12"};{\ar "12";"14"};%
{\ar "1";"3"};{\ar "3";"5"};{\ar "5";"7"};{\ar "7";"9"};{\ar "9";"11"};{\ar "11";"13"};{\ar "13";"16"};%
{\ar@/_1.5pc/"2";"6"};{\ar "4";"10"};{\ar "8";"15"};%
\end{xy}
\end{minipage}
\begin{minipage}[t]{.10\linewidth}
\begin{xy}
(10,40)*+{\ldots}="g";
\end{xy}
\end{minipage}
}
\end{figure}

\begin{figure}[H]
\centering
\resizebox{.6\width}{.6\height}{
\begin{minipage}[t]{.25\linewidth}
\begin{xy}
(15,60)*+{t_{0,1}^{(-1)}}="1";%
(15,50)*+{t_{0,2}^{(-3)}}="2"; (30,50)*+{t_{1,0}^{(-1)}}="3";%
(0,40)*+{\fbox{$t_{1,4}^{(-11)}$}}="4"; (15,40)*+{t_{0,3}^{(-5)}}="5";%
(15,30)*+{t_{0,4}^{(-7)}}="6"; (30,30)*+{t_{2,0}^{(-5)}}="7";%
(0,20)*+{t_{2,2}^{(-11)}}="8"; (15,20)*+{t_{0,5}^{(-9)}}="9";%
(15,10)*+{t_{0,6}^{(-11)}}="10"; (30,10)*+{t_{3,0}^{(-9)}}="11";%
(0,0)*+{t_{3,2}^{(-15)}}="12";(15,0)*+{t_{0,7}^{(-13)}}="13";%
(0,-15)*+{\vdots}="14";(15,-15)*+{\vdots}="15";(30,-15)*+{\vdots}="16";(15,-22)*+{}="17";%
(15,-25)*+{(d)}="n";%
{\ar "2";"1"};{\ar "5";"2"};{\ar "6";"5"};{\ar "9";"6"};{\ar "10";"9"};{\ar "13";"10"};{\ar "15";"13"};%
{\ar "12";"8"};{\ar "4";"8"};{\ar "11";"7"};{\ar "7";"3"};{\ar "16";"11"};{\ar "14";"12"};%
{\ar "2";"4"};{\ar "4";"6"};{\ar "15";"12"};%
{\ar "1";"3"};{\ar "3";"5"};{\ar "5";"7"};{\ar "7";"9"};{\ar "9";"11"};{\ar "11";"13"};{\ar "13";"16"};%
{\ar "10";"4"};{\ar "8";"15"};{\ar "8";"2"};{\ar@/_1.5pc/"6";"10"};{\ar "12";"17"};%
\end{xy}
\end{minipage}
\begin{minipage}[t]{.25\linewidth}
\begin{xy}
(15,60)*+{t_{0,1}^{(-1)}}="1";%
(15,50)*+{t_{0,2}^{(-3)}}="2"; (30,50)*+{t_{1,0}^{(-1)}}="3";%
(0,40)*+{t_{1,4}^{(-11)}}="4"; (15,40)*+{t_{0,3}^{(-5)}}="5";%
(15,30)*+{t_{0,4}^{(-7)}}="6"; (30,30)*+{t_{2,0}^{(-5)}}="7";%
(0,20)*+{\fbox{$t_{2,4}^{(-15)}$}}="8"; (15,20)*+{t_{0,5}^{(-9)}}="9";%
(15,10)*+{t_{0,6}^{(-11)}}="10"; (30,10)*+{t_{3,0}^{(-9)}}="11";%
(0,0)*+{t_{3,2}^{(-15)}}="12";(15,0)*+{t_{0,7}^{(-13)}}="13";%
(0,-15)*+{\vdots}="14";(15,-15)*+{\vdots}="15";(30,-15)*+{\vdots}="16";(15,-22)*+{}="17";%
(15,-25)*+{(e)}="n";%
{\ar "2";"1"};{\ar "5";"2"};{\ar "6";"5"};{\ar "9";"6"};{\ar "10";"9"};{\ar "13";"10"};{\ar "15";"13"};%
{\ar "8";"12"};{\ar "8";"4"};{\ar "11";"7"};{\ar "7";"3"};{\ar "16";"11"};{\ar "14";"12"};%
{\ar "4";"6"};%
{\ar "1";"3"};{\ar "3";"5"};{\ar "5";"7"};{\ar "7";"9"};{\ar "9";"11"};{\ar "11";"13"};{\ar "13";"16"};%
{\ar "10";"4"};{\ar "15";"8"};{\ar "2";"8"};{\ar@/_1.5pc/"6";"10"};{\ar "4";"15"};{\ar "12";"2"};{\ar "12";"17"};%
\end{xy}
\end{minipage}
\begin{minipage}[t]{.25\linewidth}
\begin{xy}
(15,60)*+{t_{0,1}^{(-1)}}="1";%
(15,50)*+{t_{0,2}^{(-3)}}="2"; (30,50)*+{t_{1,0}^{(-1)}}="3";%
(0,40)*+{t_{1,4}^{(-11)}}="4"; (15,40)*+{t_{0,3}^{(-5)}}="5";%
(15,30)*+{t_{0,4}^{(-7)}}="6"; (30,30)*+{t_{2,0}^{(-5)}}="7";%
(0,20)*+{t_{2,4}^{(-15)}}="8"; (15,20)*+{t_{0,5}^{(-9)}}="9";%
(15,10)*+{t_{0,6}^{(-11)}}="10"; (30,10)*+{t_{3,0}^{(-9)}}="11";%
(0,0)*+{\fbox{$t_{3,4}^{(-19)}$}}="12";(15,0)*+{t_{0,7}^{(-13)}}="13";%
(0,-15)*+{\vdots}="14";(15,-15)*+{\vdots}="15";(30,-15)*+{\vdots}="16";(15,-22)*+{}="17";%
(15,-25)*+{(f)}="n";%
{\ar "2";"1"};{\ar "5";"2"};{\ar "6";"5"};{\ar "9";"6"};{\ar "10";"9"};{\ar "13";"10"};{\ar "15";"13"};%
{\ar "12";"8"};{\ar "8";"4"};{\ar "11";"7"};{\ar "7";"3"};{\ar "16";"11"};{\ar "12";"14"};%
{\ar "4";"6"};%
{\ar "1";"3"};{\ar "3";"5"};{\ar "5";"7"};{\ar "7";"9"};{\ar "9";"11"};{\ar "11";"13"};{\ar "13";"16"};%
{\ar "10";"4"};{\ar "15";"8"};{\ar@/_1.5pc/"6";"10"};{\ar "4";"15"};{\ar "2";"12"};{\ar "8";"17"};{\ar "14";"2"};%
\end{xy}
\end{minipage}
\begin{minipage}[t]{.10\linewidth}
\begin{xy}
(10,40)*+{\ldots}="g";
\end{xy}
\end{minipage}
}
\caption{The mutation sequence $(C_3, C_3)$.} \label{mutation sequence B33}
\end{figure}
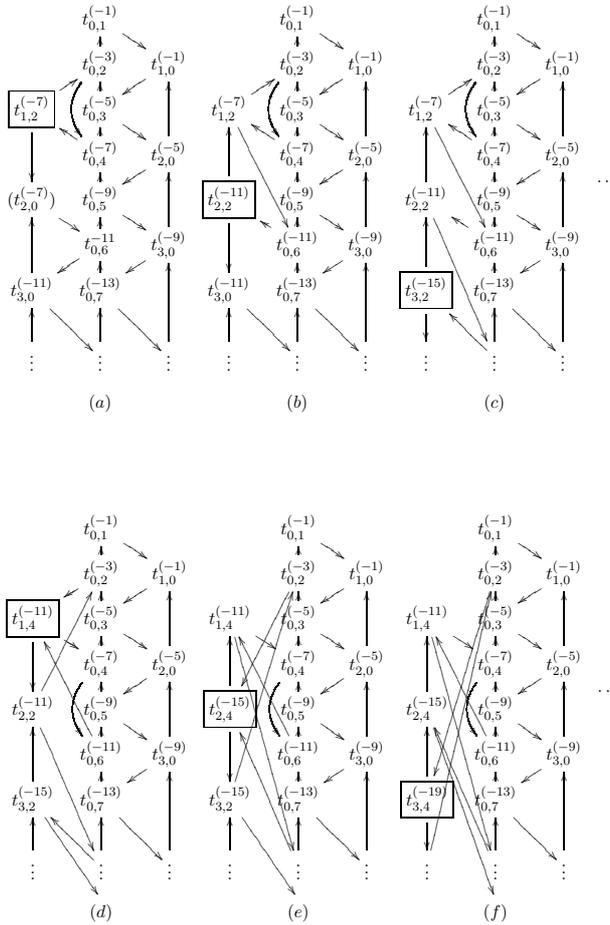

\end{document}